\documentclass[11pt]{article}

\usepackage[USenglish]{babel}
\usepackage[T1]{fontenc} 
\usepackage[utf8]{inputenc}

\usepackage{amsmath,amsthm,amssymb,amsfonts}
\usepackage{dsfont}
\usepackage{mathtools}
\usepackage{mathrsfs}

\usepackage{authblk}						% authors block

\usepackage[font={small}]{caption}

\mathtoolsset{showonlyrefs}  % uncomment to tag only the referred equations (post-production)

\usepackage{enumitem}
\usepackage{lmodern}

\usepackage[bottom=3cm, top=3cm, left=3.5cm, right=3.5cm]{geometry}

\usepackage{subfigure} 

\usepackage[bookmarks=true]{hyperref}
\usepackage{xcolor}
\hypersetup{
    colorlinks,
    linkcolor={red!50!black},
    citecolor={blue!50!black},
    urlcolor={blue!80!black},
}

\newcommand{\numberset}{\mathbb}
\newcommand{\N}{\numberset{N}}
\newcommand{\Z}{\numberset{Z}}
\newcommand{\R}{\numberset{R}}
\newcommand{\D}{\mathcal{D}}

\newcommand{\A}{\mathcal{A}}

\newcommand{\grad}{\nabla}

\newcommand{\sr}{sub-Rie\-man\-nian }
\newcommand{\hei}{\numberset{H}}
\newcommand{\eps}{\varepsilon}

\newcommand{\om}[1]{\Omega(#1)}
\newcommand{\omprime}[1]{\Omega'(#1)}
\newcommand{\bdom}[1]{\partial\Omega(#1)}
\newcommand{\deltasr}{\Delta}
\newcommand{\diverg}{\mathrm{div}_\omega}
\newcommand{\source}{f}
\newcommand{\funspace}{C_c^\infty(\omprime{\smallpar})}
\newcommand{\smallpar}{{r_0}}

\newcommand{\dvol}{d\hspace{.05em}\omega}
\newcommand{\area}{d\hspace{.05em}\sigma(y)}
\newcommand{\normal}{N}
\newcommand{\tmin}{t_{\mathrm{min}}}
\newcommand{\cut}[1]{\mathrm{Cut}\left(#1\right)}

\DeclareMathOperator{\vol}{Vol}

\DeclareMathOperator{\spn}{span}

\theoremstyle{plain}
\newtheorem{thm}{Theorem}[section]
\newtheorem{cor}[thm]{Corollary}
\newtheorem{lem}[thm]{Lemma}
\newtheorem*{lem*}{Lemma}
\newtheorem{prop}[thm]{Proposition}
\newtheorem*{prob}{Open problem}

\theoremstyle{definition}
\newtheorem{defn}[thm]{Definition}

\theoremstyle{remark}
\newtheorem{rmk}[thm]{Remark}

\title{Heat content asymptotics for sub-Riemannian manifolds}
\date{\today}

\author[1]{Luca Rizzi}
\author[1,2]{Tommaso Rossi}
\affil[1]{Univ. Grenoble Alpes, CNRS, Institut Fourier, 38000 Grenoble, France}
\affil[2]{SISSA, Via Bonomea, 265, 34136 Trieste, Italy}

\begin{document}
\maketitle 

\begin{abstract}
We study the small-time asymptotics of the heat content of smooth non-char\-acteristic domains of a general rank-varying sub-Riemannian structure, equipped with an arbitrary smooth measure. By adapting to the sub-Riemannian case a technique due to Savo, we establish the existence of the full asymptotic series:
\begin{equation}
Q_\Omega(t) = \sum_{k=0}^{\infty} a_k t^{k/2}, \qquad \text{as } t\to 0.
\end{equation}
We compute explicitly the coefficients up to order $k=5$, in terms of sub-Riemannian invariants of the domain. Furthermore, we prove that every coefficient can be obtained as the limit of the corresponding one for a suitable Riemannian extension. 

As a particular case we recover, using non-probabilistic techniques, the order $2$ formula recently obtained by Tyson and Wang in the Heisenberg group \cite{TW-heat-cont-hei}. A consequence of our fifth-order analysis is the evidence for new phenomena in presence of characteristic points. In particular, we prove that the higher order coefficients in the asymptotics can blow-up in their presence. A key tool for this last result is an exact formula for the distance from a specific surface with an isolated characteristic point in the Heisenberg group, which is of independent interest.
\end{abstract}

\tableofcontents

\section{Introduction}

Let $(M,g)$ be a complete Riemannian manifold, and $\Omega \subset M$ be a relatively compact open domain with smooth boundary. Consider the solution $u(t,x)$ of the heat equation with Dirichlet boundary conditions and homogeneous initial datum:
\begin{equation}\label{eq:he}
\begin{aligned}
(\partial_t -\Delta)u(t,x)  & =  0, & \qquad &\forall (t,x) \in (0,\infty) \times \Omega, \\
u(t,x)& = 0,  &  \qquad & \forall (t,x) \in (0,\infty) \times \partial \Omega,\\
u(0,x) & 	=  1,  & \qquad & \forall x \in \Omega,
\end{aligned}
\end{equation}
where $\Delta$ is the Laplace-Beltrami operator of $(M,g)$. The Riemannian heat content of $\Omega$ is the function
\begin{equation}
Q_\Omega(t) = \int_{\Omega} u(t,x) d\mu_g(x), \qquad t \in [0,\infty),
\end{equation}
where $d\mu_g$ is the Riemannian measure. From a physical viewpoint, $Q_\Omega(t)$ represents the total heat contained in $\Omega$ at time $t$, corresponding to a uniform initial temperature distribution, and where the boundary $\partial \Omega$ is kept at zero temperature. It turns out that $Q_\Omega(t)$ admits an asymptotic expansion as a function of $\sqrt{t}$ whose coefficients encode geometrical information about $\Omega$ and its boundary.

For Euclidean domains $\Omega \subset \R^n$, the asymptotics of $Q_\Omega(t)$ at order $1$ was computed in \cite{vdB-D}, and up to order $2$ in \cite{vdB-LG}, using probabilistic methods\footnote{Here and throughout the paper, the order is computed as a power-series in the variable $\sqrt{t}$.}. In particular, under the condition that $\partial \Omega$ is of class $C^3$, it holds:
\begin{equation}\label{eq:asympt-vdB-LG}
Q_\Omega(t) = \vol(\Omega) -\sqrt{\frac{4t}{\pi}}\sigma(\partial \Omega) + \frac{t}{2} \int_{\partial \Omega} H d\sigma + O(t^{3/2}),
\end{equation}
where $\mathrm{Vol}$ here denotes the Lebesgue measure, $\sigma$ is the corresponding surface measure on $\partial \Omega$, and $H$ is the mean curvature of $\partial \Omega$. A first non-flat case was studied in \cite{vdB-hemisphere}, where the authors computed the heat content asymptotics to order $2$ for the upper hemisphere, exploiting the explicit knowledge of the heat kernel.

For smooth domains in a Riemannian manifold, the existence of an asymptotic expansion in $\sqrt{t}$ at arbitrary order was established in \cite{vdB-G1}, where the authors also computed all coefficients up to order $4$. In this case, the volume, the perimeter, and the mean curvature appearing in \eqref{eq:asympt-vdB-LG} are replaced by their corresponding Riemannian counterparts, while the subsequent terms involve the second fundamental form of $\partial \Omega$ and the Riemann curvature tensor. We stress that the existence of a full asymptotic series is non-trivial, as the heat content is not a smooth function of $\sqrt{t}$ around $t=0$ (one can easily verify this fact by computing the heat content of a Euclidean segment). Van den Berg and Gilkey's method in \cite{vdB-G1}, which heavily exploits the functorial properties of the coefficients and invariance theory for the Riemannian curvature, has been extended to compute the heat content asymptotics up to order $5$ in \cite{vdB-G2}, and to the case of Neumann boundary conditions, see \cite{vdB-D-G,D-G}.

\subsection{Sub-Riemannian heat content asymptotics}

In this paper we study the asymptotics of the heat content in sub-Riemannian geometry. The latter is a vast generalization of Riemannian geometry, where a smoothly varying metric is defined only on a subset of preferred directions $\D_x \subseteq  T_x M$ at each point $x\in M$ (called horizontal directions). For example, $\D$ can be a sub-bundle of the tangent bundle, but we will consider the most general case of rank-varying distributions. Under the so-called H\"ormander condition, $M$ is horizontally-path connected, and the usual length-minimization procedure yields a well-defined metric. In this case, the Laplace-Beltrami operator is generalized by the sub-Laplacian $\deltasr$, which is a non-elliptic and hypoelliptic second order differential operator of H\"ormander-type \cite{Hormander,Strichartz}.

The study of the heat content asymptotics in the sub-Riemannian setting is interesting for several reasons. Firstly, there is no analogue of Levi-Civita connection, curvature, or invariance theory for a general sub-Riemannian structure. These were fundamental tools for the study of the Riemannian problem by Van den Berg and Gilkey, and hence new methods must be used in the sub-Riemannian setting. Secondly, in the general sub-Riemannian case, there is no canonical choice of measure. For this reason, we must work with a general smooth measure $\omega$, which is necessary for the definition of the sub-Laplacian\footnote{Under appropriate regularity conditions for the distribution of horizontal directions, one can define the canonical Popp's measure \cite{montgomerybook,nostropopp}, extending the Riemannian one. This construction is not possible for non-equiregular structures, e.g.\ rank-varying ones.}. Thirdly, the study of the sub-Riemannian heat content can improve our understanding of the intrinsic geometry of hypersurfaces, which is well-developed only for the case of the Heisenberg group \cite{Pauls,AF-normal,AF-hessian,CDPT-Heisenberg,Balogh-Steiner,Balogh-Gauss} and Carnot groups \cite{DGN-calculushyper} (see also \cite{veloso} for a concept of Gaussian curvature for surfaces in three-dimensional contact structures, generalizing \cite{Balogh-Gauss}). Lastly, a genuinely new phenomenon occurs in the sub-Riemannian case: characteristic points, where the distribution is tangent to $\partial \Omega$, and whose presence is source of subtle technical problems.

The study of the small-time heat content asymptotics in the sub-Riemannian setting was initiated recently by Tyson and Wang, in \cite{TW-heat-cont-hei}, where they studied the first Heisenberg group $\hei$. There, they established the existence of a small-time asymptotic series up to order $2$ in $\sqrt{t}$, for non-characteristic domains. The approach in \cite{TW-heat-cont-hei} is probabilistic, based on the interpretation of the solution of the Dirichlet problem in terms of the exit time of the corresponding Markov process. This relation holds up to an error of order $o(t)$, cf.\ \cite[Prop.\ 3.2]{TW-heat-cont-hei}, preventing the access to higher order terms in the heat content asymptotics. %, via this technique.

We use here a different method with respect to that of Tyson and Wang, by adapting a technique developed in the Riemannian case by Savo \cite{Savo-heat-cont-asymp,Savo-mean-value-lemma,Savo-gradient}. This method allows us to prove the existence of an asymptotic expansion at arbitrary order, for non-characteristic domains of general rank-varying sub-Rieman\-nian structures. Our first main result is the following. Precise definitions can be found in Section \ref{sec:prel}. 
\begin{thm}\label{t:1intro}
Let $M$ be a sub-Riemannian manifold, equipped with a smooth measure $\omega$, and let $\Omega \subset M$ be an open relatively compact subset whose boundary is smooth and has not characteristic points. Then, there exist $a_k \in \R$ such that for all $m \geq 4$ it holds
\begin{equation}
Q_\Omega(t) = \omega(\Omega) -\sqrt{\frac{4t}{\pi}}\sigma(\partial \Omega) + \frac{t}{2} \int_{\partial \Omega} H d\sigma + \sum_{k=3}^{m-1} a_k t^{k/2} + O(t^{m/2}), \qquad \text{as $t\to 0$},
\end{equation}
where $\sigma$ is the sub-Riemannian measure induced by $\omega$ on $\partial\Omega$, and $H$ is the sub-Rieman\-nian mean curvature\footnote{We recall that letting $\nu$ the outward pointing horizontal normal of $\partial\Omega$, the induced sub-Riemannian measure $\sigma$ on $\partial\Omega$ is the smooth and positive measure whose density is $|i_{\nu}\omega|_{\partial\Omega}$. Furthermore the sub-Riemannian mean curvature of $\partial\Omega$ is given by $H=\diverg(\nu)|_{\partial\Omega}$, where $\diverg(\cdot)$ is the divergence of a vector field, computed with respect to the measure $\omega$.} of $\partial \Omega$.
\end{thm}

In order to report the first few coefficients, we introduce the operator $N$, acting on smooth functions in a neighborhood of $\partial \Omega$, given by
\begin{equation}
N\phi  = 2g(\nabla \phi,\nabla \delta) + \phi \deltasr \delta,
\end{equation}
where $g$ is the sub-Riemannian scalar product, $\nabla$ is the sub-Riemannian gradient, $\deltasr = \diverg \circ \nabla$ is the sub-Laplacian (symmetric with respect to the smooth measure $\omega$), and $\delta$ is the sub-Riemannian distance from $\partial\Omega$ (which, in absence of characteristic points, is smooth on a neighborhood of $\partial \Omega$). We remark that for the sub-Riemannian horizontal mean curvature it holds $H = -\deltasr\delta|_{\partial\Omega}$.
\begin{prop}\label{p:coefficients}
With the assumptions and notations of Theorem \ref{t:1intro}, for $k\geq 1$, we have $a_k=-\int_{\partial \Omega} D_k(1)d\sigma$, where $D_k$ is an homogeneous polynomial of degree $k-1$ in the operators $\deltasr$ and $N$. In particular, it holds\begin{align}
a_0 & = \omega(\Omega), & 	 a_1 & = - \sqrt{\frac{4}{\pi}} \sigma(\partial \Omega), \\
a_2 & =  -\frac{1}{2} \int_{\partial \Omega } \deltasr \delta d\sigma, &  a_3 & = -\frac{1}{6\sqrt{\pi}}\int_{\partial \Omega} N\deltasr \delta d\sigma, \\
a_4 & = -\frac{1}{16}\int_{\partial \Omega} \deltasr^2 \delta d\sigma, & a_5 & = \frac{1}{240\sqrt{\pi}}\int_{\partial \Omega}(N^3-8N\deltasr)\deltasr \delta d\sigma.
\end{align}
For all $k\geq 1$, the operators $D_k$ are defined recursively in \eqref{eqn:recurs_op1}--\eqref{eqn:recurs_op3}.
\end{prop}
\begin{rmk}
The integrands of $a_1$ and $a_2$ have classical interpretation as the perimeter and the mean curvature of $\partial\Omega$. We observe that the integrand of $a_3$ is the so-called \emph{effective potential}, a quantity introduced in \cite{PRS-QC,FPR-sing-lapl} to describe the essential self-adjointness properties of sub-Laplacians, in presence of singular measures.
\end{rmk}
\begin{rmk}
The operators $D_k$ belong to the algebra generated by $\Delta$ and $\normal$, and are homogeneous of degree $k-1$ in the generators, where $\normal$ has degree $1$ and $\Delta$ has degree $2$. Since $\Delta(1)=0$ and $\normal(1) = \Delta\delta$, for $k\geq 2$, we have that each $a_k=-\int_{\partial\Omega}\tilde{D}_k(\deltasr\delta)d\sigma$, where $\tilde{D}_k$ is an homogeneous polynomial of degree $k-2$ in $\deltasr$ and $N$. In particular, each $a_k$ depends on the horizontal mean curvature of $\partial\Omega$ and its derivatives.
\end{rmk}
\begin{rmk}
The iterative construction of operators $D_k$, which is quite involved, has been implemented in the software \emph{Mathematica} in \cite{script}, thanks to which the coefficients $a_k$ can be immediately computed.
\end{rmk}
Before presenting further results, let us comment the proof of Theorem \ref{t:1intro}. Savo's method amounts to study the quantity
\begin{equation}
F(t,r) = \int_{\Omega(r)} u(t,x) d\omega(x),
\end{equation}
where $\Omega(r) = \{\delta > r\}$. Upon appropriate localization to deal with the non-smoothness of $\Omega(r)$ for large $r$, it turns out that $F(t,r)$ satisfies a non-homogeneous one-dimensional heat equation on the half-line $[0,\infty)$, with Neumann boundary condition at the origin. Then, the whole asymptotics of Theorem \ref{t:1intro} and the expression of the coefficients are obtained by iterating the corresponding Duhamel's formula. Some non-trivial modifications must be implemented to adapt this technique to the sub-Riemannian setting. For example, the Li-Yau estimate for the heat kernel of Riemannian manifolds with Ricci curvature bounded from below are no longer available (sub-Riemannian manifolds have, in a sense, Ricci curvature unbounded from below). Another important ingredient is the description of tubular neighborhoods of $\partial \Omega$. If, for the Heisenberg group, this can be achieved through the explicit formulas for geodesics as done in \cite{AF-normal,AF-hessian,R-tub-neigh,AFM17}, we must use a different approach for the general case, based on the study of the Hamiltonian flow on the annihilator bundle of $\partial \Omega$.

We also remark that the same method can be used, with no modifications, to study the heat content associated with a non-uniform initial condition $\phi \in C^\infty(\Omega) \cap L^2(\Omega,\omega)$. In this case, using the same notation of Proposition \ref{p:coefficients}, one has $a_k=-\int_{\partial\Omega}D_k(\phi) d\sigma$.

\subsection{Riemannian approximations}

Any sub-Riemannian structure can be obtained as a monotonic limit of Riemannian ones. This approximation scheme can be easily implemented for constant-rank distributions. In this case, a natural approximating sequence is obtained by taking any Riemannian metric $g$ extending the sub-Riemannian one, and rescaling it by a factor $1/\varepsilon$ in the transverse directions. This construction yields a one-parameter family of Riemannian structures $g_\varepsilon$. The associated Riemannian distance $d_\varepsilon$ converges, uniformly on compact sets, to the sub-Riemannian one $d_{\mathrm{SR}}$. Outside of the sub-Riemannian cut locus, one can actually prove that $d_\varepsilon \to d_{\mathrm{SR}}$ in the $C^\infty$ topology, see for example \cite{BGMR-comparison}. For totally geodesic foliations, this scheme is known under the name of \emph{canonical variation} \cite{B-einstein}. We remark though that the Riemannian curvature of the approximating sequence is unbounded below, posing some technical difficulties when taking the limit.

We introduce in this paper a generalization of the canonical variation scheme which works for general rank-varying sub-Riemannian structures. Our second result relates the coefficients of the small-time asymptotics of the Riemannian heat content $Q_\Omega^\varepsilon(t)$ of the approximating structure with the sub-Riemannian ones.
\begin{thm}\label{t:2intro}
Let $M$ be a sub-Riemannian manifold, equipped with a smooth measure $\omega$, and let $\Omega \subset M$ be an open relatively compact subset whose boundary is smooth and has not characteristic points. Then, there exists a family of Riemannian metrics $g_\varepsilon$ such that $d_{\varepsilon} \to d_{\mathrm{SR}}$ uniformly on compact sets of $M$, and such that
\begin{equation}\label{eqn:introlimit}
\lim_{\varepsilon \to 0} a_k^\varepsilon = a_k, \qquad \forall\, k \in \mathbb{N},
\end{equation}
where $a_k$ and $a_k^\varepsilon$ denote the coefficients of the sub-Riemannian small-time heat content asymptotics, and the corresponding ones for the Riemannian approximating structure.
\end{thm}
Even though $Q_\Omega^\varepsilon(t) \to Q_\Omega(t)$ in the $C^\infty$ uniform topology on compact subsets of $(0,\infty)$, this fact alone does not imply \eqref{eqn:introlimit}. A direct proof of Theorem \ref{t:2intro} would require (i) an a-priori proof of the existence of the small-time sub-Riemannian asymptotics for $Q_\Omega(t)$ and (ii) a delicate inversion of the order of the two limits $\varepsilon\to 0$ and $t \to 0$. It is also important to stress that Theorem \ref{t:1intro} is not proved using an approximation scheme and thus \emph{it is not} a consequence of Theorem \ref{t:2intro}. The latter will be rather proved using the explicit iterative formula for the coefficients of the Riemannian and sub-Riemannian heat content expansions.

\subsection{Characteristic points}

One main assumption in all our results is that $\partial\Omega$ does not contain characteristic points. This is quite restrictive for the case of Heisenberg group, where the only non-characteristic domains are homeomorphic to a torus. More generally, for any contact sub-Riemannian manifold, the non-characteristic assumption and the contact structure imply that $\partial \Omega$ must have vanishing Euler characteristic. On the other hand, the non-characteristic assumption is less restrictive for general structures: it is not hard to prove that for any smooth manifold $M$ of dimension $n\geq 4$, and any smooth relatively compact domain $\Omega$ with smooth boundary $\partial \Omega$, there exists a possibly rank-varying sub-Riemannian structure on $M$ such that $\partial \Omega$ has no characteristic points.

The non-characteristic assumption is crucial for the smoothness of the distance from $\partial\Omega$ and for the existence of smooth tubular neighborhoods, cf.\ Theorem \ref{thm:sr_tub_neigh}. Furthermore, even if the existence of solutions in $L^2$ to the heat equation with Dirichlet boundary conditions holds from general spectral theory (and thus $Q_\Omega(t)$ is well-defined), their smoothness up to the boundary may fail close to characteristic points \cite{Jerison-KohnI,Jerison-KohnII}.

Despite all these difficulties, one might wonder whether the small-time heat content asymptotic formula of Theorem \ref{t:1intro} makes sense, at least formally, for domains with characteristic points. Firstly, we note that if $\Sigma$ is a smooth embedded hypersurface in a sub-Riemannian manifold $M$, then the set of characteristic points have zero measure in $\Sigma$, see \cite{Balogh-size}. Secondly, any choice of smooth measure $\omega$ on $M$ induces a smooth surface measure $\sigma$ on $\Sigma$, even in presence of characteristic points. It is sufficient to consider the contraction of $\omega$ with the horizontal unit normal to $\Sigma$ (see for example \cite[Section 8]{DGN-calculushyper} for the case of Carnot groups, where this notion is related with the horizontal perimeter measure). Thirdly, the sub-Riemannian mean curvature is locally integrable with respect $\sigma$, even in presence of characteristic points, cf.\ \cite{DGN-Integrability,integrabilityH}. As a consequence of all these facts, all terms appearing in the order $2$ formula for $\hei$ in \cite{TW-heat-cont-hei} are well-defined also for characteristic domains. This seems to suggest that the same small-time asymptotic formula might hold also for characteristic domains. Our analysis shows that this cannot be true at higher order.

\begin{thm}\label{t:3intro}
Let $\hei$ be the first Heisenberg group, and consider the plane $\Sigma = \{z = 0\}$. Observe that the origin is an isolated characteristic point. Denote with $\sigma$ the sub-Riemannian surface measure on $\Sigma$ induced by the Lebesgue measure on $\hei$. Then the integrand of the coefficient $a_5$ of the small-time heat content expansion is not locally integrable with respect to $\sigma$ around the characteristic point of $\Sigma$.
\end{thm}

Theorem \ref{t:3intro} shows that the asymptotic formula of Theorem \ref{t:1intro} is false at order $k\geq 5$ for domains with characteristic points. In the example of Theorem \ref{t:3intro}, it turns out that the integrands of the coefficients $a_3$ and $a_4$ are still locally integrable with respect to the sub-Riemannian surface measure. We expect, however, that one can build a less symmetric example where also the integrand of $a_4$ is not integrable close to a characteristic point, cf.\ Remark \ref{rmk:a4}. On the other hand, Theorem \ref{t:1intro} might still be true at lower order (it has already been remarked that the coefficients appearing therein remain well-defined for characteristic domains in $\hei$ up to $k=2$).
\begin{prob}
Is it true that, for smooth domains in $\hei$ with characteristic points, the asymptotic expansion of Theorem \ref{t:1intro} remains valid up to some order $0< k < 5$?
\end{prob}

To prove Theorem \ref{t:3intro} we derived an exact formula for the sub-Riemannian distance from the $xy$-plane in $\hei$. To our best knowledge, this is the first time such an explicit global formula appears in the literature, and it has independent interest. For example, it can be used to study the loss of regularity of the distance at characteristic points.

\begin{thm}\label{t:explicitformulaintro}
The distance from the $xy$-plane in the first Heisenberg group $\hei$, for all $p\in z$-axis, is given by
\begin{equation}
\delta(p) = \sqrt{2\pi |z_p|},
\end{equation}
while for all $p\notin z$-axis, it is given in cylindrical coordinates by
\begin{equation}
\delta(p) = r_p \frac{4\xi_p+y_0(\xi_p)}{\sqrt{1+y_0(\xi_p)^2}}, \qquad \xi_p = \frac{|z_p|}{r_p^2},
\end{equation}
where $\xi \mapsto y_0(\xi)$ is the unique smooth function such that
\begin{equation}
4\xi+y_0+(1+y_0^2)\arctan(y_0) = 0.
\end{equation}
\end{thm}
\begin{rmk}
The proof of Theorem \ref{t:explicitformulaintro} consists in a non-trivial characterization of minimal geodesics to the $xy$-plane, and the corresponding cut-locus. Local formulas are easier to obtain, cf.\ \cite[Ex.\ 5.1]{Balogh-Steiner}, using the well-known minimality property of short segments of normal \sr geodesics. However, these local formulas and related estimates do not hold uniformly when approaching characteristic points. %On the other hand, Theorem \ref{t:explicitformulaintro} is valid for all $p \in \hei$, even those lying arbitrarily close to the origin. This feature is key to the proof of Theorem \ref{t:3intro}.
\end{rmk}

\paragraph{Relative heat content.} In this paper we focused on the heat content of a domain $\Omega$ of a sub-Riemannian manifold $M$. A related concept is that of \emph{relative} heat content of a domain $\Omega$ in $M$, which is obtained by considering, instead of the Dirichlet problem \eqref{eq:he}, the solution to the heat equation on the whole manifold with initial condition $u(0,x) = \mathds{1}_{\Omega}(x)$. In other words, in terms of the heat kernel $p_t^M(x,y)$ of $M$, the relative heat content is the function
\begin{equation}
H_{\Omega}(t) = \int_{\Omega \times \Omega} p_t^M(x,y)d\omega(x)d\omega(y), \qquad t>0.
\end{equation}
This object was studied mainly for the relation between the small-time asymptotics of $H_\Omega(t)$ and the perimeter of $\Omega$, as already identified by Ledoux \cite{MR1309086} for subsets of the Euclidean space (cf.\ also \cite{MR3116054} for a comparison with the heat content, and \cite{MR3019137} for higher-order asymptotics). Concerning non-Euclidean structures, we mention that the relation between relative heat content and perimeter was studied in \cite{MR2972544} for the case of step 2 Carnot groups, and more generally in \cite{MR3558354} for metric spaces supporting a Poincaré inequality. Finally, in \cite{MR3022730}, the authors establish a low order small-time asymptotics for the heat evolution of non-characteristic graphs over Carnot groups, and study its connection with the horizontal mean curvature flow.

\paragraph{Structure of the paper.} In Section \ref{sec:prel} we recall the basic definitions of \sr geometry, and in particular we explain how to approximate any \sr structure with a Riemannian one. In Section \ref{sec:loc_bd} and \ref{sec:mean_value}, we provide the key ingredients to obtain the heat content asymptotic, i.e., the localization to the boundary of the heat content and the \sr version of the mean value lemma, respectively. Then, in Section \ref{sec:asymptotic}, we develop the asymptotic expansion of the heat content, concluding the proof of Theorem \ref{t:1intro} and Proposition \ref{p:coefficients}. In Section \ref{sec:coeff_conv}, we show how to obtain the coefficients through a Riemannian approximation, proving Theorem \ref{t:2intro}. Finally, in Section \ref{sec:integrability} we show that, in presence of a characteristic point, the integrand of the coefficient $a_5$ is not locally integrable, proving Theorem \ref{t:3intro}.

\paragraph{Acknowledgments.} This work was supported by the Grants ANR-15-CE40-0018, ANR-18-CE40-0012 of the ANR, and the Project VINCI 2019 ref.\ c2-1212.

A previous proof of Theorem \ref{thm:not_feel_bd} employed Riemannian approximations and the classical Li-Yau heat kernel estimate. We thank Emmanuel Trelat for stimulating discussions that led to the current simpler proof, using instead a hypoelliptic version of Kac's principle. The previous proof, including a more detailed appendix on Riemannian approximations of the heat semi-group, is available on the first preprint version of this paper. We also thank Luca Capogna for informative discussions on the regularity of the Dirichlet problem up to the boundary in the sub-Riemannian case.

		% introduction
\section{Preliminaries}\label{sec:prel}
We recall some essential facts in \sr geometry, following \cite{ABB-srgeom}.

\subsection{Sub-Riemannian geometry}
Let $M$ be a smooth, connected finite-dimensional manifold. A \sr structure on $M$ is defined by a set of $N$ global smooth vector fields $X_1,\ldots,X_N$, called a \emph{generating frame}. The generating frame defines a \emph{distribution} of subspaces of the tangent spaces at each point $x\in M$, given by
\begin{equation}
\label{eqn:gen_frame}
\D_x=\spn\{X_1(x),\ldots,X_N(x)\}\subseteq T_xM.
\end{equation}
We assume that the distribution is \emph{bracket-generating}, i.e. the Lie algebra of smooth vector fields generated by $X_1,\dots,X_N$, evaluated at the point $x$, coincides with $T_x M$, for all $x\in M$. The generating frame induces a norm on the distribution at $x$, namely
\begin{equation}
\label{eqn:induced_norm}
g_x(v,v)=\inf\left\{\sum_{i=1}^Nu_i^2\mid \sum_{i=1}^Nu_iX_i(x)=v\right\},\qquad\forall\,v\in\D_x,
\end{equation}
which, in turn, defines an inner product on $\D_x$ by polarization. We use the shorthand $\|\cdot\|_g$ for the corresponding norm. We say that $\gamma : [0,T] \to M$ is a \emph{horizontal curve}, if it is absolutely continuous and
\begin{equation}
\dot\gamma(t)\in\D_{\gamma(t)}, \qquad\text{for a.e.}\,t\in [0,T].
\end{equation}
This implies that there exists $u:[0,T]\to\R^N$, such that
\begin{equation}
\dot\gamma(t)=\sum_{i=1}^N u_i(t) X_i(\gamma(t)), \qquad \text{for a.e.}\, t \in [0,T].
\end{equation} 
Moreover, we require that $u\in L^2([0,T],\R^N)$. If $\gamma$ is a horizontal curve, then the map $t\mapsto \|\dot\gamma(t)\|_g$ is measurable on $[0,T]$, hence integrable \cite[Lemma 3.12]{ABB-srgeom}. We define the \emph{length} of a horizontal curve as follows:
\begin{equation}
\ell(\gamma) = \int_0^T \|\dot\gamma(t)\|_g dt.
\end{equation}
The \emph{\sr distance} is defined, for any $x,y\in M$, by
\begin{equation}\label{eq:infimo}
d_{\mathrm{SR}}(x,y) = \inf\{\ell(\gamma)\mid \gamma \text{ horizontal curve between $x$ and $y$} \}.
\end{equation}
By Chow-Rashevskii Theorem, the bracket-generating assumption ensures that the distance $d_{\mathrm{SR}}\colon M\times M\to\R$ is finite and continuous. Furthermore it induces the same topology as the manifold one.
\begin{rmk}
The above definition includes all classical constant-rank sub-Riemannian structures as in \cite{montgomerybook,Riffordbook} (where $\D$ is a vector distribution and $g$ a symmetric and positive tensor on $\D$), but also general rank-varying sub-Riemannian structures. The same \sr structure can arise from different generating families. 
\end{rmk}

\subsection{Geodesics and Hamiltonian flow}
\label{sec:geod}

We recall some basic facts about length-minimizing curves, and the Hamiltonian formalism, used in Sections \ref{sec:coeff_conv}-\ref{sec:integrability}.

A \emph{geodesic} is a horizontal curve $\gamma :[0,T] \to M$, parametrized with constant speed, and such that any sufficiently short segment is length-minimizing. The \emph{sub-Riemannian Hamiltonian} is the smooth function $H : T^*M \to \R$, given by
\begin{equation}
\label{eq:Hamiltonian}
H(\lambda) := \frac{1}{2}\sum_{i=1}^N \langle \lambda, X_i \rangle^2, \qquad \lambda \in T^*M,
\end{equation}
where $X_1,\ldots,X_N$ is a generating frame for the sub-Riemannian structure, and $\langle \lambda, \cdot \rangle $ denotes the action of covectors on vectors. The \emph{Hamiltonian vector field} $\vec H$ on $T^*M$ is then defined by $\varsigma(\cdot,\vec H)=dH$, where $\varsigma\in\Lambda^2(T^*M)$ is the canonical symplectic form.

Solutions $\lambda : [0,T] \to T^*M$ of the \emph{Hamilton equations}
\begin{equation}\label{eq:Hamiltoneqs}
\dot{\lambda}(t) = \vec{H}(\lambda(t)),
\end{equation}
are called \emph{normal extremals}. Their projections $\gamma(t) = \pi(\lambda(t))$ on $M$, where $\pi:T^*M\to M$ is the bundle projection, are locally length-minimizing horizontal curves parametrized with constant speed, and are called \emph{normal geodesics}. If $\gamma$ is a normal geodesic with normal extremal $\lambda$, then its speed is given by $\| \dot\gamma \|_g = \sqrt{2H(\lambda)}$. In particular
\begin{equation}
\label{eq:speed}
\ell(\gamma|_{[0,t]}) = t \sqrt{2H(\lambda(0))},\qquad \forall\, t\in[0,T].
\end{equation} 
The \emph{exponential map} $\exp_{x} : T_x^*M \to M$, with base $x \in M$ is
\begin{equation}
\exp_x (\lambda) = \pi \circ e^{\vec{H}}(\lambda), \qquad \lambda \in T_x^*M,
\end{equation}
where $e^{\vec{H}}$ denotes the flow of $\vec{H}$, which we assume to be well-defined up to time $1$. This is the case, for example, when $(M,d_{\mathrm{SR}})$ is a complete metric space.

There is another class of length-minimizing curves in sub-Riemannian geometry, called \emph{abnormal} or \emph{singular}. As for the normal case, to these curves it corresponds an extremal lift $\lambda(t)$ on $T^*M$, which however may not follow the Hamiltonian dynamics \eqref{eq:Hamiltoneqs}. 
Here we only observe that an abnormal extremal lift $\lambda(t)\in T^*M$ satisfies
\begin{equation}
\label{eq:abn}
\langle \lambda(t),\D_{\pi(\lambda(t))}\rangle=0\quad \text{and} \quad \lambda(t)\neq 0,\qquad \forall\, t\in[0,T] ,
\end{equation}
that is $H(\lambda(t))\equiv 0$. A geodesic may be abnormal and normal at the same time. 

\paragraph{Length-minimizers to the boundary.} Consider now a closed embedded submanifold $S \subset M$ of positive codimension (we will only need the case in which $S=\partial\Omega$ is the smooth boundary of a relatively compact open subset $\Omega\subset M$). Let $\gamma:[0,T]\to M$ be a horizontal curve, parametrized with constant speed, such that $\gamma(0)\in S$, $\gamma(T) = x \in M\setminus S$, and such that it minimizes the distance to $S$, that is
\begin{equation}
\ell(\gamma)=\inf\{d_{\mathrm{SR}}(z,x)\mid z\in  S\}.
\end{equation}
In particular, $\gamma$ is a geodesic. Any corresponding normal or abnormal lift, say $\lambda :[0,T]\to T^*M$, must satisfy the transversality conditions \cite[Thm 12.4]{AS-GeometricControl}
\begin{equation}\label{eq:trcondition}
\langle \lambda(0), v\rangle=0,\qquad \forall \,v\in T_{\gamma(0)} S,
\end{equation}
in other words, the initial covector $\lambda(0)$ must belong to the annihilator bundle $\A(S) = \{\lambda \in T^*M \mid \langle \lambda, T_{\pi(\lambda)} S\rangle = 0\}$ of $S$. An immediate consequence for the case of codimension $1$ is the following, cf.\ \cite[Prop.\ 2.7]{FPR-sing-lapl}.

\begin{prop}\label{prop:no_abn}
Consider a sub-Riemannian structure on a smooth manifold $M$. Let $S \subset M$ be a closed embedded hypersurface. Let $\gamma:[0,T]\to M$ be a horizontal curve such that $\gamma(0)\in  S$, $\gamma(T)=p\in M\setminus S$, and that minimizes the distance to $S$. Then $\gamma(0)\in S$ is a characteristic point if and only if $\gamma$ is abnormal.
\end{prop}

In particular, Proposition \ref{prop:no_abn} implies that as soon has $S$ has no characteristic points, curves which are length-minimizing to $S$ are all normal, that is $\gamma(t) =\exp_x(t\lambda)$, for some unique (up to reparametrization) $\lambda$ respecting \eqref{eq:trcondition}.

\subsection{The heat content}

Let $M$ be a \sr manifold. Let $\omega$ be a smooth measure on $M$, defined by a positive tensor density.  The \emph{divergence} of a smooth vector field is defined by
\begin{equation}
\diverg (X)\omega=\mathcal{L}_X \omega, \qquad\forall\,X\in\Gamma(TM),
\end{equation}
where $\mathcal{L}_X$ denotes the Lie derivative in the direction of $X$. The \emph{horizontal gradient} of a function $f\in C^\infty(M)$, denoted by $\grad f$, is defined as the horizontal vector field (i.e. tangent to the distribution at each point), such that
\begin{equation}
g_x(\grad f(x),v)= v(f)(x),\qquad\forall\,v\in\D_x,
\end{equation} 
where $v$ acts as a derivation on $f$. In terms of a generating frame as in \eqref{eqn:gen_frame}, one has
\begin{equation}
\grad f=\sum_{i=1}^NX_i(f)X_i,\qquad\forall\,f\in C^\infty(M).
\end{equation}
We recall the divergence theorem (we stress that $M$ is not required to be orientable):
\begin{equation}
\label{eqn:diverg_thm}
\int_{\partial\Omega}fg(X,\nu)d\sigma=\int_{\Omega}\left(f\diverg X+g(\grad f,X)\right)\dvol,
\end{equation}  
for any smooth function $f$ and vector field $X$. In \eqref{eqn:diverg_thm}, $\nu$ is the outward-pointing vector field to $\partial\Omega$ and $\sigma$ is the induced \sr measure on $\partial\Omega$ (i.e.\ the one whose density is $\sigma=|i_\nu\omega|_{\partial\Omega}$).

The \emph{sub-Laplacian} is the operator $\deltasr= \diverg\circ\grad$, acting on $C^\infty(M)$. Again, we may write its expression with respect to a generating frame \eqref{eqn:gen_frame}, obtaining
\begin{equation}
\deltasr f=\sum_{i=1}^N\left\{X^2_i(f)+X_i(f)\diverg (X_i)\right\},\qquad\forall\,f\in C^\infty(M).
\end{equation}
Let $\Omega\subset M$ be an open relatively compact set with smooth boundary. This means that the closure $\bar{\Omega}$ is a compact manifold with smooth boundary. We consider the \emph{Dirichlet problem for the heat equation} on $\Omega$, that is we look for functions $u$ such that
\begin{equation}\label{eqn:dir_prob}
\begin{aligned}
(\partial_t -\Delta)u(t,x)  & =  0, & \qquad &\forall (t,x) \in (0,\infty) \times \Omega, \\
u(t,x)& = 0,  &  \qquad & \forall (t,x) \in (0,\infty) \times \partial \Omega,\\
u(0,x) & 	=  1,  & \qquad & \forall x \in \Omega.
\end{aligned}
\end{equation}
We denote by $L^2(\Omega,\omega)$, or simply by $L^2$, the space of real functions on $\Omega$ which are square-integrable with respect to the measure $\omega$. Notice that we can represent the solution to \eqref{eqn:dir_prob}, as
\begin{equation}
u(t,\cdot)=e^{t\Delta}\mathds{1}_\Omega, \qquad\forall\,t\geq 0,
\end{equation}
where $e^{t\Delta}\colon L^2\rightarrow L^2$ denotes the semi-group generated by the Dirichlet self-adjoint extension of the sub-Laplacian on $\Omega$. We remark that for all $\varphi \in L^2$, the function $e^{t\Delta} \varphi$ is smooth for all $(t,x) \in (0,\infty) \times \Omega$, by hypoellipticity of the heat operator.

\begin{defn}[Heat Content]
Let $u(t,x)$ be the solution to \eqref{eqn:dir_prob}. We define the \emph{heat content}, associated with $\Omega$, as
\begin{equation}
Q_\Omega(t)=\int_\Omega{u(t,x)\dvol(x)}, \qquad \forall\,t>0.
\end{equation}
\end{defn}

We recall here two properties of the heat semi-group that we will use in the sequel. First of all, the solution to \eqref{eqn:dir_prob} satisfies a weak maximum principle, meaning that
\begin{equation}
\label{eqn:wmax_prin}
0\leq u(t,x)\leq 1, \qquad\forall\,x\in\Omega,\ \forall\,t>0.
\end{equation}
Second of all, the domain monotonicity property holds. Let $\Omega'\subset M$ be a relatively compact domain, such that $\Omega\subset\Omega'$.  Denoting with $\Delta'$ the Dirichlet sub-Laplacian on $\Omega'$, for any $\varphi\in L^2(\Omega,\omega)$, we have
\begin{equation}
\label{eqn:domain_mon}
e^{t\Delta}\varphi(x)\leq e^{t\Delta'}\varphi(x), \qquad\forall x\in\Omega,\ \forall\,t>0,
\end{equation}
where $e^{t\Delta'}$ is the semi-group generated by $\Delta'$ on $L^2(\Omega',\omega)$. Properties \eqref{eqn:wmax_prin} and \eqref{eqn:domain_mon} can be proven following the blueprint of the Riemannian proofs (see \cite[Thm. 5.11, Thm. 5.23]{MR2569498}). An alternative proof can be given using Riemannian approximations (see Section \ref{sec:riem_approx} for details) exploiting the respective properties for the Riemannian heat semi-groups and then passing to the limit.

\begin{defn}
We say that $x\in\partial\Omega$ is a \emph{characteristic point}, or tangency point, if the distribution is tangent to $\partial \Omega$ at $x$, that is
\begin{equation}\label{eqn:char_pts}
\D_x\subseteq T_x(\partial\Omega).
\end{equation}
\end{defn}
We will assume that $\partial\Omega$ has no characteristic points. We say in this case that $\Omega$ is a non-characteristic domain. In this case, the solution $u(t,x)$ of the heat equation with Dirichlet boundary conditions and initial datum $\phi \in C^\infty(\Omega)$ exists, is unique, and is smooth on $(0,\infty)\times \bar{\Omega}$, see \cite[Thm.\ 2.5]{GM-Green}. This is a consequence of the analogous result for the stationary problem considered in \cite{KohnNirenberg-noncoercive}. In the Riemannian case, this is a classical textbook result, see \cite[Sec.\ 7, Thm.\ 7]{Evans}.

\subsection{Approximation via Riemannian structures}
\label{sec:riem_approx}
We describe an approximation procedure of a \sr metric structure via a family of Riemannian ones, extending the classical canonical variation scheme, which we will use in Section \ref{sec:coeff_conv}. See also \cite{CaCi} for a different approximation scheme.

Let $(\D,g)$ be a \sr structure on $M$. Consider a global generating frame for a Riemannian structure, i.e.\ a set of $L$ global vector fields $\tilde{X}_1,\ldots,\tilde{X}_L$ such that
\begin{equation}
\label{eqn:gen_frame2}
T_xM=\spn\{\tilde{X}_1(x),\ldots,\tilde{X}_L(x)\},\qquad\forall\,x\in M.
\end{equation}
For any $\eps\in\R$, consider the following family of global smooth vector fields:
\begin{equation}
\label{eqn:gen_frame3}
\left\{X_1,\ldots,X_N,\eps \tilde{X}_1,\ldots,\eps \tilde{X}_L\right\}.
\end{equation}
This is a global generating frame for the whole tangent space at each point, therefore, it induces a scalar product on $T_xM$, defined, as in the \sr case, by formula \eqref{eqn:induced_norm}, which we denote by $g_\eps$. Since \eqref{eqn:gen_frame2} holds, $g_\eps$ is a Riemannian metric on $M$ and the corresponding Riemannian distance is denoted by $d_\eps$. Furthermore it holds
\begin{equation}
\label{eqn:conv_dist}
d_\eps\xrightarrow{\eps\to 0}d_{\mathrm{SR}}, \qquad\text{uniformly on the compact sets of }M.
\end{equation}
An explicit proof of this fact in the constant-rank case can be found in \cite[Lemma A.1]{BGMR-comparison}. The same proof holds verbatim in the general rank-varying case, replacing local orthonormal frames with the generating frame \eqref{eqn:gen_frame3}, and replacing the controls of minimizing geodesics with the minimal controls, whose definition can be found in \cite[Ch. 3]{ABB-srgeom}.

We may call the 1-parameter family of metric spaces $\{(M,d_\eps)\}_{\eps\in\R}$ a \emph{Riemannian variation} of $(M,d_{\mathrm{SR}})$. For the fixed smooth measure $\omega$ we define the corresponding gradient and the Laplacian, denoted respectively by $\nabla_\eps$ and $\Delta_\eps$. The expression of $\nabla_\eps$ acting on $f\in C^\infty(M)$, with respect to the generating frame \eqref{eqn:gen_frame3}, is given by
\begin{equation}
\label{eqn:approx_grad}
\nabla_\eps f=\sum_{i=1}^NX_i(f)X_i+\eps^2\sum_{\alpha=1}^L\tilde{X}_\alpha(f)\tilde{X}_\alpha=\grad f+\eps^2 \tilde{\nabla}f,\qquad\forall\,f\in C^\infty(M),
\end{equation} 
while the expression of $\Delta_\eps$ is given by
\begin{align}
\Delta_\eps f &= \sum_{i=1}^N\left\{X^2_i(f)+X_i(f)\diverg (X_i)\right\}+\eps^2\sum_{\alpha=1}^L\left\{\tilde{X}^2_\alpha(f)+\tilde{X}_\alpha(f)\diverg (\tilde{X}_\alpha)\right\}\\
							&= \deltasr f+\eps^2\tilde{\Delta}f.\label{eqn:approx_lapl}
\end{align}
Here, $\nabla$ and $\Delta$ represent the \sr gradient and sub-Laplacian, respectively, while $\tilde{\nabla}$ and $\tilde{\Delta}$ represent the gradient and the Laplacian, computed with respect to the Riemannian structure defined by the frame in \eqref{eqn:gen_frame2}, respectively. Fix now an open relatively compact subset $\Omega\subset M$ with smooth boundary. The sequence of Riemannian structures yields naturally an approximation of the solution $u(t,x)$ of the Dirichlet problem \eqref{eqn:dir_prob} on $\Omega$. To this purpose, define
\begin{equation}
D=\{ f\in C^\infty(\bar\Omega) \mid f|_{\partial\Omega} = 0\}.
\end{equation}
The set $D$ is a common core for the $\Delta$ and $\Delta_\eps$. This is a well-known fact in the Euclidean case, but the proof (e.g.\ see \cite[Prop.\ 1, p.\ 264]{MR0493421}) holds unchanged in the (sub-)Riemannian case, by the spectral theorem. Using the explicit expressions of $\Delta$ and $\Delta_\eps$ on smooth functions, in terms of a generating frame for the Riemannian variation \eqref{eqn:approx_lapl}, we get that
\begin{equation}
\Delta_\eps\varphi\xrightarrow{\eps\to 0}\Delta\varphi,\qquad\forall\,\varphi\in D,
\end{equation}
uniformly on $\Omega$, and hence also on $L^2$. Consequently, applying Trotter-Kato's Theorem (see implication $(a)\Rightarrow(d)$ in \cite[Thm.\ 4.8, Ch.\ III]{EN-semigroupsbook}), we obtain
\begin{equation}
\label{eqn:conv_heat}
e^{t\Delta_\eps}\varphi\xrightarrow[\|\cdot\|_{L^2}]{\eps\to 0}e^{t\deltasr}\varphi,\qquad\forall\,\varphi\in L^2, \text{ uniformly in }t\in[0,T],
\end{equation}
where $e^{t\Delta_\eps}$ and $e^{t\Delta}$, for $t>0$, denote the semi-groups associated with the Dirichlet self-adjoint extension of $\Delta_\eps$ and $\Delta$, respectively. Integrating \eqref{eqn:conv_heat}, with $\varphi =1|_{\Omega}$, we obtain an analogous result for the heat content:
\begin{equation}
\label{eqn:conv_cont}
Q^\eps_\Omega(t)\xrightarrow{\eps\to 0}Q_\Omega(t), \qquad \text{uniformly on } [0,T], 
\end{equation}
where $Q_\Omega^\eps(t)$ denotes the Riemannian heat content of $\Omega$ for the metric $g_\eps$.		% preliminaries
\section{Localization to the boundary of the heat content}
\label{sec:loc_bd}
The next result is an application of a hypoelliptic version of Kac's principle of not feeling the boundary \cite{MR1325580}. It states that solutions of the heat equation, at any interior point, are not influenced by the boundary conditions, up to a uniform $O(t^\infty)$ term. As we detail in the proof, the uniformity we need can be achieved through off-diagonal estimates for H\"ormander-type operators as in \cite{JS-estimates}. 

\begin{thm}
\label{thm:not_feel_bd}
Let $M$ be a sub-Riemannian manifold, equipped with a smooth measure $\omega$, and let $\Omega \subset M$ be an  open relatively compact subset with smooth boundary. Then for any compact set $K\subset \Omega$ it holds
\begin{equation}\label{eqn:sr_bd}
1-u(t,x) = O(t^\infty) \qquad \text{as} \ t\to 0, \quad \text{ uniformly for $x \in K$},
\end{equation}
where $u(t,x)$ denotes the solution to \eqref{eqn:dir_prob}.
\end{thm}

\begin{rmk}
Here, the boundary $\partial\Omega$ can contain characteristic points.
\end{rmk}

\begin{proof}
Without loss of generality we can assume that $M$ is compact, as any modification of the structure outside $\Omega$ does not change $u(t,x)$. Denote by $p_t^M(x,y)$ and $p_t^\Omega(x,y)$ the heat kernels on $M$ and $\Omega$, respectively (for the case of $\Omega$ recall that we impose Dirichlet boundary conditions). We have
\begin{equation}\label{eq:Kac1}
1-u(t,x) = \int_{M\setminus\Omega} p_t^M(x,y) d\omega(y) + \int_\Omega \left(p_t^M(x,y)-p_t^\Omega(x,y)\right)d\omega(y).
\end{equation}
Let $K$ as in the statement, and recall that in \eqref{eq:Kac1} $x\in K$. 

To estimate the first term of \eqref{eq:Kac1} we use the off-diagonal estimate for H\"ormander-type operators in \cite[Prop.\ 3]{JS-estimates}, that is
\begin{equation}\label{prop3Jerison}
p^M_t(x,y) \leq C_a e^{-c_a/t}, \qquad \forall\,x,y \text{ with } d(x,y) \geq a,\quad t< 1,
\end{equation}
for some positive constants $C_a,c_a>0$. Since in that first integral of \eqref{eq:Kac1} one has $d(x,y)\geq d(K,\partial \Omega)=a>0$, we conclude that
\begin{equation}
\int_{M\setminus\Omega} p_t^M(x,y) d\omega(y) \leq \vol(M) C_a e^{-c_a/t} = O(t^\infty), \qquad \text{ uniformly for $x \in K$}.
\end{equation}
To estimate the second term of \eqref{eq:Kac1}, let $K'\subset \Omega$ be another compact subset such that $K \subset \mathring{K}' \subset \Omega$, so that $d(K,\Omega\setminus K')=b >0$. We split the second term of \eqref{eq:Kac1} as
\begin{multline}\label{eq:Kac2}
\int_\Omega \left(p_t^M(x,y)-p_t^\Omega(x,y)\right)d\omega(y) = \int_{K'} \left(p_t^M(x,y)-p_t^\Omega(x,y)\right)d\omega(y) \\ + \int_{\Omega\setminus K'} \left(p_t^M(x,y)-p_t^\Omega(x,y)\right)d\omega(y).
\end{multline}
By domain monotonicity \eqref{eqn:domain_mon}, we have $0\leq p_t^\Omega(x,y) \leq  p_t^M(x,y)$, whence
\begin{align}
\left|\int_{\Omega\setminus K'} \left(p_t^M(x,y)-p_t^\Omega(x,y)\right)d\omega(y)\right| & \leq 2 \int_{\Omega\setminus K'} p_t^M(x,y) d\omega(y) \\
& \leq 2 \vol(\Omega) C_b e^{-c_b/t} = O(t^\infty),
\end{align}
uniformly for $x\in K$, using again \eqref{prop3Jerison}. For the first term in \eqref{eq:Kac2}, we use a hypoelliptic version of Kac's principle (cf.\ \cite[p.\ 836]{JS-estimates} or \cite[Thm.\ 3.1]{YHT-2}). This means, in particular, that
\begin{equation}
p_t^M(x,y)-p_t^\Omega(x,y) = O(t^\infty), \qquad \forall\, x,y \in K', \quad t <1,
\end{equation}
uniformly (we stress that this estimate cannot hold uniformly on $\bar\Omega$, whence the reason for the restriction to compact sets). We conclude easily.
\end{proof}
\section{Sub-Riemannian mean value lemma}
\label{sec:mean_value}
%In this section we prove a \sr version of the mean value lemma. 
Let $M$ be a \sr manifold with smooth measure $\omega$, and let $\Omega$ be a relatively compact subset of $M$ with smooth boundary, and assume that $\partial\Omega$ has no characteristic points. Denote with $\delta(\cdot)= d_{\mathrm{SR}}(\partial\Omega,\cdot)$ and define, for $r>0$, the open set (see figure \ref{fig:omega_r}) 
\begin{equation}
\om{r}=\{x\in \Omega \mid \delta(x)>r\}.
\end{equation}

\begin{figure}[ht]
\centering
\includegraphics[height=4cm]{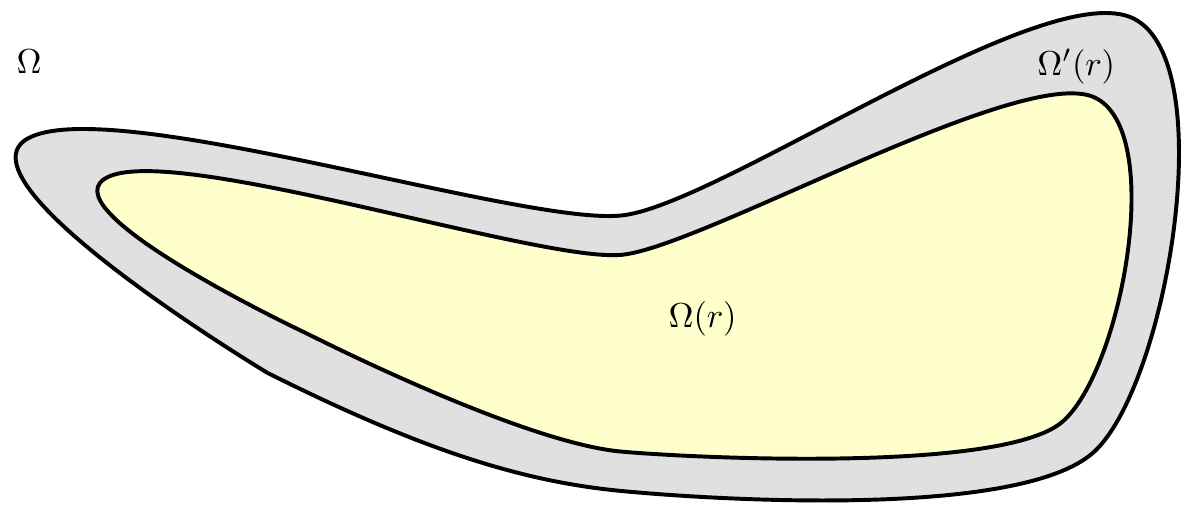}
\caption{The set $\om{r}$ is highlighted in yellow. The gray set is $\omprime{r}=\bar\Omega\setminus\om{r}$.}
\label{fig:omega_r}
\end{figure}
Notice that, in general, the distance function is only $1$-Lipschitz, therefore the boun\-dary of $\om{r}$, which we denote by $\bdom{r}$, may not be a smooth submanifold. Then, for $r\in [0,\infty)$ and $v\in C^\infty(\Omega)$, consider the function 
\begin{equation}
\label{eqn:mean_val_fun}
F(r)= \int_{\om{r}}v(x)\dvol(x).
\end{equation}
The function $F$ represents the mean value on $\Omega(r)$.

\begin{thm}
\label{thm:sr_mean_val}
Let $M$ be a \sr manifold, equipped with a smooth measure $\omega$, and let $\Omega\subset M$ be an open relatively compact subset of $M$ whose boundary is smooth and has not characteristic points. Let $\delta\colon\bar\Omega\to[0,\infty)$ be the distance function from $\partial\Omega$. Then there exists $\smallpar>0$ such that the function $F$, defined as in \eqref{eqn:mean_val_fun}, is smooth on $[0,\smallpar)$ and, for $r<\smallpar$
\begin{equation}
\label{eqn:2der_gap}
F''(r)=\int_{\Omega(r)}{\deltasr v(x)\dvol(x)}-\int_{\bdom{r}}{v(y)\deltasr\delta(y)\area},
\end{equation}
where $\deltasr$ is the sub-Laplacian associated with $\omega$ and $\sigma$ is the induced measure on $\bdom{r}$.
\end{thm}

Before giving the proof, we establish some basic properties of the distance function $\delta$ for general \sr structures. We recall here the result we need, proved in \cite[Prop. 3.1]{FPR-sing-lapl}.

\begin{thm}[Tubular Neighborhood]
\label{thm:sr_tub_neigh}
Let $M$ be a \sr manifold, e\-quipped with a smooth measure $\omega$, and let $\Omega\subset M$ be an open relatively compact subset of $M$  whose boundary is smooth and has not characteristic points. Let $\delta : \bar\Omega \to [0,\infty)$ be the distance function from $\partial \Omega$.
Then, we have:
\begin{itemize}
\item[i)] $\delta$ is Lipschitz with respect to the \sr distance and $\|\grad\delta\|_g\leq 1$ a.e.;
\item[ii)] there exists $\smallpar>0$ such that $\delta\colon \omprime{\smallpar}\to[0,\infty)$ is smooth, where $\omprime{r}=\bar\Omega\setminus\om{r}$;
\item[iii)] there exists a smooth diffeomorphism $G: [0,\smallpar)\times \partial\Omega \to \omprime{\smallpar}$, such that
\begin{equation}
\label{tub_neig_main_prop}
\delta(G(t,y))=t\quad\text{and}\quad G_*\partial_t=\grad\delta,\qquad \forall\,(t,y)\in [0,\smallpar)\times \partial\Omega.
\end{equation}
Moreover, $\|\grad \delta\|_g\equiv 1$ on $\omprime{\smallpar}$.
\end{itemize}
\end{thm}
The assumption on the characteristic set of $\partial\Omega$ is crucial to guarantee the smoothness of the distance around the boundary. However, if $\partial\Omega$ has characteristic points, $\delta$ is not even Lipschitz in coordinates, see \cite{ACS_reg} for more details.

\begin{proof}[Proof of Theorem \ref{thm:sr_mean_val}]
Thanks to Theorem \ref{thm:sr_tub_neigh}, we get a co-area type formula in the \sr case: it is enough to compute the expression of the volume in the coordinates induced by the diffeomorphism $G$. Namely, we obtain
\begin{equation}
\label{eqn:sr_coarea}
\int_{\omprime{r}}{v(x)\dvol(x)}=\int_0^r{\int_{\bdom{s}}{v(y)\area ds}},
\end{equation}
where $\sigma$ denotes the measure induced by $\omega$ on $\bdom{s}$ via the diffeomorphism $G$. Since $G_* \partial_s =\grad\delta$, for the corresponding smooth densities, we have $\sigma = |i_\nu\omega|$, where $\nu = -\grad\delta$. Formula \eqref{eqn:sr_coarea} holds as long as the function $G$ is a smooth diffeomorphism, i.e. for any $r\leq\smallpar$. Consequently, in the interval $[0,\smallpar)$, we can compute the first derivative of the function $F$ as follows:
\begin{align}
F(r) &=\int_{\om{r}}{v(x)\dvol(x)}=\int_{\Omega}{v(x)\dvol(x)}-\int_{\omprime{r}}{v(x)\dvol(x)}\\
		 &=\int_{\Omega}{v(x)\dvol(x)}-\int_0^r\int_{\bdom{s}}{v(y)\area  ds}.
\end{align}
Then, we have
\begin{equation}
\label{eqn:1st_derF}
F'(r) =\frac{\partial}{\partial r}\left(\int_{\Omega}{v(x)\dvol(x)}-\int_0^r\int_{\bdom{s}}{v(y)\area  ds}\right)=-\int_{\bdom{r}}{v(y) \area}.
\end{equation} 
For the second derivative of $F$, let us rewrite its first derivative, using the divergence theorem\eqref{eqn:diverg_thm} on $\om{r}$, and the fact that $\|\grad\delta\|_g=1$:
\begin{equation}
F'(r)=-\int_{\bdom{r}}v d\sigma =-\int_{\bdom{r}}vg(\nu,\nu)d\sigma=-\int_{\om{r}}\left(v\diverg\nu+g(\grad v,\nu)\right)\dvol.
\end{equation}
Using again the fact that $\omprime{r} = \bar\Omega \setminus \om{r}$, and applying the co-area formula to the previous expression, we obtain
\begin{align}
F''(r) & =- \frac{\partial}{\partial r} \int_{\om{r}}\left(v\diverg\nu+g(\grad v,\nu)\right)\dvol \\
& = - \frac{\partial}{\partial r} \left(\int_\Omega \left(v\diverg\nu+g(\grad v,\nu)\right)\dvol - \int_{\omprime{r}}\left(v\diverg\nu+g(\grad v,\nu)\right)\dvol\right) \\
&=-\frac{\partial}{\partial r}\left(-\int_0^r\int_{\bdom{s}}\left(v\diverg\nu+g(\grad v,\nu)\right)d\sigma ds\right)\\ 
&=\int_{\bdom{r}}\left(-v\deltasr\delta+g(\grad v,\nu)\right)d\sigma\\
			&=\int_{\om{r}}\deltasr v(x)\dvol(x)-\int_{\bdom{r}}{v(y)\deltasr\delta(y)\area},
\end{align}
where, in the last passage, we have used again the divergence theorem. 
\end{proof}

\begin{cor}
\label{cor:nonhom_heat}
Under the hypotheses of Theorem \ref{thm:sr_mean_val}, the function $F$, defined as 
\begin{equation}
\label{eqn:defn_F}
F(t,r)=\int_{\om{r}}u(t,x)\dvol(x),
\end{equation}
where $u(t,x)$ denotes the solution to the heat equation on $\Omega$ \eqref{eqn:dir_prob}, satisfies the non-homogeneous one-dimensional heat equation
\begin{equation}
\label{eqn:nonhom_heat}
(\partial_t-\partial_r^2)F(t,r)=\int_{\bdom{r}}{u(t,y)\deltasr\delta(y) \area}, \qquad t > 0,\quad r\in[0,\smallpar),
\end{equation}
with Neumann boundary condition $\partial_r F(t,0) =0$.
\end{cor}

%\begin{rmk}
%In Corollary \ref{cor:nonhom_heat} the inversion of derivation in $t$ and integral sign is justified by the dominated convergence Theorem, since all the functions involved are smooth for $t>0$ and the set of integration is compact. This will always be the case throughout the rest of the article.
%\end{rmk}			% Sub-Riemannian mean value lemma
\section{Heat content asymptotics}
\label{sec:asymptotic}
In this section, we compute the asymptotic expansion of the \sr heat content, associated with the Dirichlet problem on a non-characteristic domain.

Since $Q_\Omega(0) = \omega(\Omega)$, we can reduce ourselves to the study of the quantity
\begin{equation}\label{eq:G}
G(t,r) = \int_{\om{r}} (1-u(t,x)) d\omega(x).
\end{equation}
The small-time asymptotics of the heat content is then recovered by the asymptotics series, as $t\to 0$, of $G(t,0)$, since we have $G(t,0)= \omega(\Omega) - Q_\Omega(t)$. Compare \eqref{eq:G} with \eqref{eqn:defn_F}. While $F$ satisfies a non-homogeneous heat equation on $[0,r_0)$, cf.\ Corollary \ref{cor:nonhom_heat}, the advantage of $G$ is that, upon localization, it satisfies a non-homogeneous heat equation on the \emph{whole} half-line. To this purpose, let $\phi,\eta:\bar{\Omega} \to \R$ smooth functions with compact support and such that
\begin{equation}
\phi + \eta  \equiv 1, \qquad \mathrm{supp}(\phi) \subset \omprime{r_0}, \qquad \mathrm{supp}(\eta) \subset \om{r_0/2},
\end{equation}
where $r_0$ is radius of the tubular neighborhood of $\partial\Omega$ on which the \sr distance $\delta :\bar\Omega \to \R$ is smooth, cf.\ Theorem \ref{thm:sr_tub_neigh}. We have then
\begin{align}
G(t,r) & = \int_{\om{r}} (1-u(t,x))\phi(x) d\omega(x) +  \int_{\om{r}} (1-u(t,x))\eta(x) d\omega(x) \\
& = \int_{\om{r}} (1-u(t,x))\phi(x) d\omega(x) +  \int_{\Omega} (1-u(t,x))\eta(x) d\omega(x) \\
& = \int_{\om{r}} (1-u(t,x))\phi(x) d\omega(x) + O(t^\infty), \label{eq:split}
\end{align}
where we used Theorem \ref{thm:not_feel_bd} to deal with the second term, having set $K=\mathrm{supp}(\eta)$. For this reason, we focus on the quantity
\begin{equation}
I\phi(t,r) = \int_{\om{r}}{(1-u(t,x))\phi(x)\dvol(x)}.
\end{equation}
Since the support of $\phi$ is contained in $\omprime{r_0}$, it turns out that $I\phi(t,r)$ is smooth for $t>0$ and $r\geq 0$, and it is compactly supported in the second variable. This discussion motivates the following definition.
\begin{defn}
\label{def:op_ILambda}
For all $t\geq 0$, we define the one-parameter families of operators $I$ and $\Lambda$ on the space $\funspace$, by
\begin{align}
I\phi(t,r) &= \int_{\om{r}}{(1-u(t,x))\phi(x)\dvol(x)},\\
\Lambda\phi(t,r) &= -\partial_rI\phi(t,r)=\int_{\bdom{r}}{(1-u(t,y))\phi(y)\area},
\end{align}
for any $\phi\in \funspace$, and where $\sigma$ denotes the induced measure on $\bdom{r}$.
\end{defn}

\begin{lem}
\label{lem:op_LI}
Let $M$ be a sub-Riemannian manifold, equipped with a smooth measure $\omega$, and let $\Omega \subset M$ be an open relatively compact subset whose boundary is smooth and has not characteristic points. Let $L= \partial_t-\partial_{r}^2$ the one-dimensional heat operator. Then, as an operator on $\funspace$
\begin{equation}
LI=\Lambda\normal+I\deltasr,
\end{equation} 
where $\normal$ is the operator defined by
\begin{equation}
\label{eqn:sr_normal}
\normal\phi= 2g\left(\grad\phi,\grad\delta\right)+\phi\deltasr\delta,\qquad\forall\,\phi\in \funspace,
\end{equation}
and  $\delta \colon \bar\Omega\to \R$ is the \sr distance function from $\partial\Omega$.
\end{lem}
\begin{proof}
The computation of $\partial_t I\phi$ is direct, while for $\partial_r^2 I\phi$ we apply Theorem \ref{thm:sr_mean_val} with the function $v = (1-u)\phi$. Then apply the divergence formula, recalling that the \sr outward unit normal to $\partial\Omega(r)$ is $\nu = -\nabla \delta|_{\partial\Omega(r)}$. We refer the reader to Lemma \ref{lem:iter_L} for a complete proof.
\end{proof}

Lemma \ref{lem:op_LI}, in particular, implies that $I\phi(t,r)$ satisfies a non-homogeneous  heat equation on the whole half-line, with Neumann boundary conditions.
\begin{cor}\label{cor:nonhom_heat_Iphi}
For any $\phi \in \funspace$, the function $I\phi$, defined as 
\begin{equation}
I\phi(t,r)=\int_{\om{r}}{(1-u(t,x))\phi(x)\dvol(x)},
\end{equation}
where $u(t,x)$ denotes the solution to the heat equation on $\Omega$ \eqref{eqn:dir_prob}, satisfies the non-homogeneous one-dimensional heat equation
\begin{equation}
(\partial_t-\partial_r^2)I\phi(t,r)= \Lambda\normal\phi(t,r)+I\deltasr\phi(t,r), \qquad t > 0,\quad r\geq 0,
\end{equation}
with Neumann boundary condition $\partial_r I\phi(t,0) =0$, for all $t\geq 0$.
\end{cor}

Thanks to Corollary \ref{cor:nonhom_heat_Iphi}, and an appropriate Duhamel-type principle, we can obtain an explicit formula for $I\phi(t,r)$, which yields its asymptotic series for $t\to 0$. The next lemma contains the general form of the Duhamel's principle that we need.

\begin{lem}[Duhamel's principle]
\label{lem:duham_prin}
Let $\source\in C((0,\infty)\times [0,\infty))$, $v_0,v_1\in C([0,\infty))$, such that $f(t,\cdot)$ and $v_0$ are compactly supported and assume that
\begin{equation}
\label{eqn:lim_source}
\exists\lim_{t\to 0} f(t,r).	
\end{equation}
Consider the non-homogeneous heat equation on the half-line:
\begin{equation}
\begin{aligned}\label{eqn:neum_prob}
Lv(t,r) & = f(t,r), & \qquad & \text{for }t>0,\ r>0, \\
v(0,r)&= v_0(r),&  \qquad  & \text{for }r>0,\\
\partial_rv(t,0)&= v_1(t),& \qquad  & \text{for }t>0,
\end{aligned}
\end{equation}
where $L=\partial_t-\partial_{r}^2$. Then, for $t>0$, the solution to \eqref{eqn:neum_prob} is given by
\begin{multline}
\label{eqn:duham_prin}
v(t,r)=\int_0^\infty e(t,r,s)v_0(s)ds +\int_0^t\int_0^\infty e(t-\tau,r,s)f(\tau,s)ds d\tau \\
-\int_0^te(t-\tau,r,0)v_1(\tau)d\tau,
\end{multline}
where $e(t,r,s)$ is the Neumann heat kernel on the half-line, that is
\begin{equation}\label{eqn:neumanheat}
e(t,r,s) = \frac{1}{\sqrt{4\pi t}}\left(e^{-(r-s)^2/4t}+e^{-(r+s)^2/4t}\right).
\end{equation}
\end{lem}

\begin{rmk}
The choice of the Neumann heat kernel in \eqref{eqn:duham_prin} is due to the fact that $I\phi(t,r)$ satisfies Neumann boundary conditions, that is $v_1=0$, and thus it simplifies the first step of the iteration for the computation of the small-time asymptotics. Subsequent steps will need the general form of \eqref{eqn:duham_prin}.
\end{rmk}

\begin{proof}[Proof of Lemma \ref{lem:duham_prin}]
We derive formula \eqref{eqn:duham_prin}, assuming that all the functions are sufficiently regular so that all steps are justified. Then, one can check that \eqref{eqn:duham_prin} holds true under the regularity assumptions in the statement, concluding the proof by uniqueness of the solution to \eqref{eqn:neum_prob}. Let us reduce problem \eqref{eqn:neum_prob} to a standard Neumann problem, defining $\tilde v(t,r)= v(t,r)-rv_1(t)$. Now, $\tilde v$ satisfies:
\begin{equation}
\begin{aligned}
(\partial_t-\partial_{r}^2)\tilde v(t,r) & = f(t,r)-rv'_1(t), & \qquad & \text{for }t>0,\ r>0, \\
\tilde v(0,r) & = v_0(r)-rv_1(0),& \qquad & \text{for }r>0,\\
\partial_r\tilde v(t,0) & = 0,& \qquad & \text{for }t>0.
\end{aligned}
\end{equation}
Therefore, we know that $\tilde v$ is given by
\begin{multline}
\tilde v(t,r)=\int_0^\infty e(t,r,s)\left(v_0(s)-sv_1(0)\right)ds \\+\int_0^t\int_0^\infty e(t-\tau,r,s)\left(f(\tau,s)-sv'_1(\tau)\right)ds d\tau.
\end{multline} 
Recalling that $v(t,r)=\tilde v(t,r)+rv_1(t)$, integrating twice by parts, and using the properties of the Neumann heat kernel $e(t,s,r)$, we obtain \eqref{eqn:duham_prin}.
\end{proof}

We can now prove Theorem \ref{t:1intro}. We first prove it at first order (Theorem \ref{thm:1ord_Q}), and then we iterate the construction obtaining the whole asymptotics (Theorem \ref{thm:full_asymp} and Proposition \ref{p:coefficients}).

\subsection{First order asymptotic expansion}

\begin{thm}
\label{thm:1ord_Q}
Let $M$ be a sub-Riemannian manifold, equipped with a smooth measure $\omega$, and let $\Omega \subset M$ be an open relatively compact subset whose boundary is smooth and has not characteristic points. The heat content $Q_\Omega(t)$ satisfies
\begin{equation}
\label{eqn:1ord_Q}
Q_\Omega(t)=\omega(\Omega)-\sqrt{\frac{4t}{\pi}}\sigma(\partial\Omega)+O(t),\qquad\text{as }t\to 0,
\end{equation}
where $\sigma$ is the sub-Riemannian measure induced by $\omega$ on $\partial\Omega$.
\end{thm}

\begin{proof}
Let $\phi\in \funspace$ and apply Lemma \ref{lem:duham_prin} to $v(t,r)=I\phi(t,r)$. On one hand, using the definitions of the operators $I$ and $\Lambda$, the initial conditions are:
\begin{align}
v_0(r) & = v(0,r) = I\phi(0,r) = 0 \\
v_1(t) & =\partial_r v(t,0) = -\Lambda\phi(t,0) = -\int_{\partial \Omega} \phi(y) d\sigma(y),
\end{align}
where we used the fact that $u(0,\cdot)=1$ on $\Omega$ and $u(t,\cdot)|_{\partial\Omega} = 0$. On the other hand, for the non-homogeneous term of the heat equation, we obtain using Lemma \ref{lem:op_LI}
\begin{equation}
f(t,r) = LI\phi(t,r) = \Lambda \normal\phi(t,r)+I\deltasr \phi(t,r).
\end{equation}
Notice that $v_0$, $f(t,\cdot)$ are compactly supported, and $v_1(t)$ is constant. Furthermore $\lim_{t\to 0} f(t,r)$ is well-defined for all $r\geq 0$ since $\normal\phi$ and $\deltasr\phi$ belong to $\funspace$. Having checked the assumptions, we obtain from \eqref{eqn:duham_prin}:
\begin{align}
I\phi(t,0)& =\frac{1}{\sqrt{\pi}}\int_{\partial\Omega}{\phi(y)\area}\int_0^t \frac{d\tau}{\sqrt{t-\tau}}+\int_0^t{\int_0^\infty{e(t-\tau,s,0)LI\phi(\tau,s)ds}d\tau}\\
& =
\sqrt{\frac{4t}{\pi}} \int_{\partial\Omega}{\phi(y)\area}+\int_0^t{\int_0^\infty{e(t-\tau,s,0)LI\phi(\tau,s)ds}d\tau}.
\end{align}
We now prove that the second term in the right-hand side is a remainder of order $O(t)$. Notice that
\begin{equation}
\left|LI\phi(t,s)\right|=\left|\Lambda \normal\phi(t,s)+I\deltasr\phi(t,s)\right|\leq C_1\|\normal\phi\|_{L^\infty(\Omega)}+C_2\|\deltasr\phi\|_{L^\infty(\Omega)}= C,
\end{equation} 
having used the inequality $0\leq 1-u(t,x)\leq 1$ by the weak maximum principle \eqref{eqn:wmax_prin}. Thus, we have
\begin{align}
\left|\int_0^t{\int_0^\infty{e(t-\tau,s,0)LI\phi(\tau,s)ds}d\tau}\right| &\leq C\int_0^t{\int_0^\infty{e(t-\tau,s,0)ds}d\tau}\\
																																												 &= C\int_0^t{\frac{1}{\sqrt{\pi (t-\tau)}}\int_0^\infty{e^{-s^2/4(t-\tau)}ds}d\tau}\\
																																												 &= Ct,
\end{align}
implying that:
\begin{equation}
\label{eqn:1st_Iphi}
I\phi(t,0)=\sqrt{\frac{4t}{\pi}}\int_{\partial\Omega}{\phi(y) \area}+O(t),\qquad\text{as }t\to 0.
\end{equation}
We now apply the argument at the beginning of Section \ref{sec:asymptotic}, by choosing $\phi$ to be a function with $\mathrm{supp}(\phi)\subset \omprime{r_0}$, and such that $\phi \equiv 1$ when restricted to $\omprime{r_0/2}$. In particular $\phi|_{\partial\Omega} =1$. In this case $\omega(\Omega) - Q_\Omega(t)$ has the same small-time asymptotics of $I\phi(t,0)$ up to an $O(t^\infty)$, and hence
\begin{equation}
Q_\Omega(t) = \omega(\Omega) - I\phi(t,0) + O(t^\infty)  = \omega(\Omega) - \sqrt{\frac{4t}{\pi}} \sigma(\partial\Omega) + O(t), \qquad \text{as } t\to 0,
\end{equation}
concluding the proof of the first-order asymptotics.
\end{proof}

\subsection{Full asymptotic expansion}\label{sec:fullasymptotic}

We would like to iterate the Duhamel's principle \eqref{eqn:duham_prin} for the function $I\phi(t,r)$ to obtain an higher order asymptotic expansion of $I\phi(t,0)$. However, the source term $f(t,r)=L^k\phi(t,r)$ for $k\geq 2$ may not satisfy the assumption \eqref{eqn:lim_source}, due to the non-differentiability of the solution $u(t,x)$ of the Dirichlet problem at $t=0$. To avoid this technical issue, one can introduce a suitable mollification of the operators $I$ and $\Lambda$, for which one can iterate Duhamel's principle at arbitrary order. There is no substantial difference with respect to the Riemannian case developed by Savo in \cite{Savo-heat-cont-asymp}. In this section we only outline the main passages leading to the proof of Theorem \ref{t:1intro} at arbitrary order, and we included more details, for self-consistence, in Appendix \ref{app:B}.

The iterations of the Duhamel's principle lead to define the following family of operators acting on smooth functions on $\bar\Omega$ with compact support around $\partial\Omega$. Set
\begin{equation}
R_{00} = \mathrm{Id}, \qquad S_{00} = 0,
\end{equation}
where $\mathrm{Id}$ denotes the identity operator. Recall the definition of $N$ in \eqref{eqn:sr_normal} and set, for all $k\geq 1$, and $0\leq j \leq k$
\begin{align}
R_{kj}& =-(\normal^2-\deltasr)R_{k-1,j}+\normal S_{k-1,j}, \label{eqn:1st_fam1}\\
S_{kj}& =\normal R_{k-1,j-1}-\deltasr\normal R_{k-1,j}+\deltasr S_{k-1,j}, \label{eqn:1st_fam2}
\end{align}
and $R_{kj}=S_{kj}=0$, for all other values of the indices, i.e.\ $k<0$, $j<0$ or $k<j$. Moreover, for $k\geq 0$, define 
\begin{equation}
\label{eqn:def_Zalpha}
Z_k=\sum_{j=0}^{k-1}\{k,j-1\}R_{k+j-1,j}\qquad\text{and}\qquad A_k=\sum_{j=0}^{k+1}\{k,j\}S_{k+j,j},
\end{equation}
with the convention that $Z_0=0$, and having set
\begin{equation}
\{k,j\}= \frac{\Gamma(k+j+1/2)}{(k+j)!\Gamma(k+1/2)},
\end{equation}
and $\Gamma$ denotes the Euler Gamma function. Finally, let $D_k$ be the operators acting on smooth functions on $\bar\Omega$ with compact support around $\partial\Omega$, defined inductively by the formulas:
\begin{align}
D_1 & =\frac{2}{\sqrt{\pi}}\mathrm{Id}, \label{eqn:recurs_op1}\\
D_{2n} &=\frac{1}{\sqrt{\pi}}\sum_{i=1}^n\frac{\Gamma(i+1/2)\Gamma(n-i+1/2)}{n!}D_{2i-1}A_{n-i}, \label{eqn:recurs_op2}\\
D_{2n+1} &=\frac{1}{\sqrt{\pi}}Z_{n+1}+\frac{1}{\sqrt{\pi}}\sum_{i=1}^n\frac{\Gamma(i+1)\Gamma(n-i+1/2)}{\Gamma(n+3/2)}D_{2i}A_{n-i}. \label{eqn:recurs_op3}
\end{align}
The coefficients of the small-time asymptotic expansion of the heat content are computed using the recursive formulas \eqref{eqn:recurs_op1}--\eqref{eqn:recurs_op3}. See \cite{script} for a \emph{Mathematica} implementation of this algorithm.
\begin{rmk}
The operators $D_k$ act on smooth functions on $\bar\Omega$, $\phi$ compactly supported around $\partial \Omega$. The restriction $D_k(\phi)|_{\partial\Omega}$ is a smooth function on $\partial\Omega$, which depends only on the germ of $\phi$ at $\partial\Omega$. Thus, the integrand $D_k(1)$ appearing in the definition of the $a_k$ in Theorem \ref{thm:full_asymp} is a slight abuse of notation to denote the action of $D_k$ on a smooth function, with compact support, such that $\phi \equiv 1$ in a neighborhood of $\partial\Omega$.
\end{rmk}

\begin{thm}
\label{thm:full_asymp}
Let $M$ be a sub-Riemannian manifold, equipped with a smooth measure $\omega$, and let $\Omega \subset M$ be an open relatively compact subset whose boundary is smooth and has not characteristic points. Then for all $m\geq 1$, the heat content $Q_\Omega(t)$ satisfies
\begin{equation}
\label{eqn:full_asymp}
Q_\Omega(t)=\sum_{k=0}^{m-1}a_kt^{k/2} + O(t^{m/2}),\qquad\text{as }t\to 0,
\end{equation}
where $a_0=\omega(\Omega)$ and $a_k=-\int_{\partial\Omega}D_k(1)d\sigma$.
\end{thm}
\begin{proof}
As in the proof of Theorem \ref{thm:1ord_Q}, we apply the argument at the beginning of Section \ref{sec:asymptotic}, by choosing $\phi$ to be a function with $\mathrm{supp}(\phi)\subset \omprime{r_0}$, and such that $\phi \equiv 1$ when restricted to $\omprime{r_0/2}$. Hence
\begin{equation}
\label{eqn:rel_QIphi}
Q_\Omega(t) = \omega(\Omega) - I\phi(t,0) + O(t^\infty),\qquad \text{as } t\to 0.
\end{equation}
It is enough to compute the asymptotic expansion of $I\phi(t,0)$.
Since $\phi\in \funspace$, we can apply the iteration of the Duhamel's formula to $I\phi(t,0)$, cf.\ Theorem \ref{thm:asymp_Iphi}, obtaining, for any $m\in\N$
\begin{equation}
\label{eqn:coeff_beta}
 I\phi(t,0)=\sum_{k=1}^m\left(\int_{\partial\Omega}D_k\phi(y)\area \right)t^{k/2}+O(t^{(m+1)/2}),\qquad\text{as }t\to 0.
\end{equation}
Since, by construction, $\phi\equiv 1$ close to $\partial\Omega$, the coefficients in \eqref{eqn:coeff_beta} simplify to
\begin{equation}
a_k= -\int_{\partial\Omega}D_k(\phi)d\sigma=-\int_{\partial\Omega}D_k(1)d\sigma, \qquad \forall\,k\in\N.
\end{equation}
We conclude the proof by replacing \eqref{eqn:coeff_beta} in \eqref{eqn:rel_QIphi}.
\end{proof}			% heat content asymptotic 
\section{Riemannian approximations and asymptotic series}
\label{sec:coeff_conv}
In this section, we show that the coefficients of the \sr heat content asymptotics can be approximated by their Riemannian counterpart (cf. Theorem \ref{t:2intro}).

Let $(\D,g)$ be a \sr structure on $M$ and fix a Riemannian variation, $\{(M,g_\eps)\}_{\eps\in \R}$, of the type explained in Section \ref{sec:riem_approx}. As we explained there, it holds
\begin{equation}
d_\eps\xrightarrow{\eps\to 0}d_{\mathrm{SR}}, \qquad\text{uniformly on the compact sets of }M.
\end{equation}
We begin with a result, of independent interest, on the corresponding approximation result for the distance from the boundary of a compact set.

\begin{lem}
\label{lem:smooth_dist}
Let $M$ be a sub-Riemannian manifold, and let $\Omega \subset M$ be an open relatively compact subset whose boundary is smooth and has not characteristic points. Let $\delta,\delta_\eps:\bar\Omega\to \R$ the \sr and $\eps$-Riemannian distances from $\partial\Omega$, and fix $\bar\eps>0$. Then, there exists $U\subset\bar{\Omega}$, neighborhood of $\partial\Omega$, such that $\delta,\delta_\eps\in C^\infty(U)$ for any $|\eps| <\bar\eps$ and $\delta_\eps\to\delta$, as $\eps\to 0$, uniformly on the compact sets of $U$, with all their derivatives.
\end{lem}
\begin{proof} 
Let us consider the annihilator bundle of the tangent bundle to $\partial\Omega$, i.e.\ the $1$-dimensional smooth vector bundle $\A(\partial\Omega)$ over $\partial\Omega$ with fibers
\begin{equation}
\A_y(\partial\Omega)=\{\lambda\in T^*_yM \mid \left\langle \lambda,T_y\partial\Omega\right\rangle=0\},\qquad \forall\, y\in\partial\Omega.
\end{equation} 
For all $\eps \in \R$, denote by $E_\varepsilon$ the restriction to $\A(\partial\Omega$) of the cotangent exponential map in the $\eps$-Riemannian approximant, namely
\begin{equation}
E_\eps(\lambda)= \exp_{\pi(\lambda)}^\eps(\lambda)=\pi\circ e^{\vec{H}_\eps}(\lambda), \qquad \lambda\in\A(\partial\Omega).
\end{equation} 
Here, $H_\eps :T^*M \to \R$ is the one-parameter family of Hamiltonians for the $\eps$-Riemannian structure, $\vec{H}_\eps$ the corresponding Hamiltonian vector field, and $e^{\vec{H}_\eps}$ the corresponding flow (cf.\ Section \ref{sec:geod}). In terms of the generating frame \eqref{eqn:gen_frame3}, we have
\begin{equation}\label{eqn:hameps}
H_\eps(\lambda)=\frac{1}{2}\sum_{i=1}^N \langle\lambda,X_i\rangle^2 + \eps^2\frac{1}{2}\sum_{i=1}^L \langle\lambda,\tilde{X}_i\rangle^2, \qquad \forall\,\lambda \in T^*M.
\end{equation}
Notice that the value $\eps=0$ corresponds to the corresponding \sr quantities, so that the subscript is omitted when $\eps =0$. 

Thanks to the non-characteristic assumption, for any $\lambda \in \A(\partial\Omega)$, with $\lambda \neq 0$, it holds $\lambda(\D) \neq 0$, and hence $H_\eps(\lambda)>0$. It follows that the $H_\eps$, for all $\eps \in \R$, are well-defined norms on the one-dimensional fibers of $\A(\partial\Omega)$. 

Define the map\footnote{To avoid complications, we assume here that all the vector fields $\vec{H}_\eps$ are complete. If this is not the case, the domain of $F$ should be replaced by a suitable neighborhood of $i(\partial\Omega)\times \R$.}
\begin{equation}
F\colon\A(\partial\Omega)\times \R \to M \times \R,\qquad F(\lambda,\eps)=(E_\eps(\lambda),\eps).
\end{equation}
Let $i:\partial\Omega\hookrightarrow \A(\partial\Omega)$ be the embedding as the set of zero covectors. Thanks to the non-characteristic assumption, one can show that $F$ has full rank on $\partial\Omega \times \R$, that is around points $(i(y),0)$. Set $I=(-\bar\eps,\bar\eps)$. It follows that there exists $a>0$ such that, letting  
\begin{equation}
V=\{\lambda \in\A(\partial\Omega) \mid \sqrt{2H(\lambda)} <a\},
\end{equation}
the map $F$ restricts to a smooth diffeomorphism from $V \times I$ onto its image. In particular, denoting by $p_1\colon M\times\R\rightarrow M$ the projection onto the first factor, each map $E_\eps(\cdot)=p_1\left(F(\cdot,\eps)\right)$ is a smooth diffeomorphism from $V$ to its image, for all $\eps \in I$. Notice that, by \eqref{eqn:hameps}, it holds $H \leq H_\eps$ for all $\eps$, so that, letting
\begin{equation}
V_\eps = \{\lambda \in \A(\partial\Omega)\mid \sqrt{2H_\eps(\lambda)} < a\}, \qquad \forall\, \eps \in I,
\end{equation}
it holds $V_\eps \subset V$ and thus $E_\eps$ maps diffeomorphically $V_\eps$ to its image $U_\eps = E_\eps(V_\eps)$, see figure \ref{fig:distanceproof}. We claim that $U_\eps = \{\delta_\eps <a\}$ and $\delta_\eps$ is smooth on $\bar\Omega \cap U_\eps$.

\begin{figure}
\centering
\includegraphics[width=\textwidth]{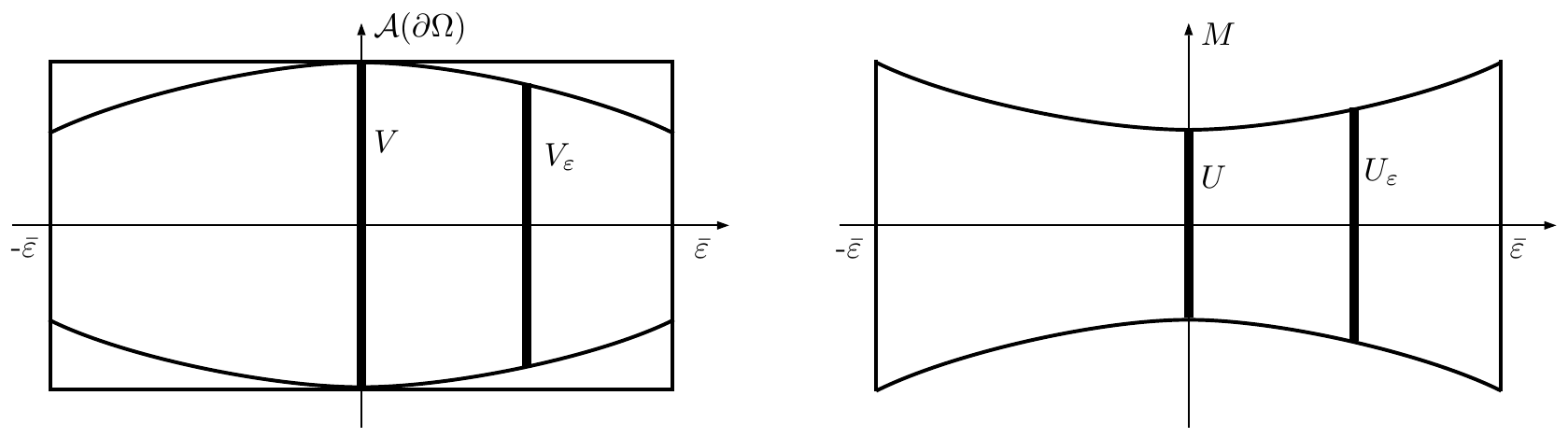}
\caption{Domain and range of the diffeomorphism $F(\cdot)$. The sections $V_\eps$ (resp.\ $U_\eps$) are monotonically non-increasing (resp.\ non-decreasing) family of sets as a function of $|\eps|$.}\label{fig:distanceproof}
\end{figure}
In order to prove the claim, fix $\eps \in I$. For $\lambda \in V_\eps$, let $\gamma_\lambda:[0,1]\to M$ be the $g_\eps$-geodesic with initial covector $\lambda$, that is $\gamma_\lambda(t)	 = \exp_{\pi(\lambda)}^\eps(t\lambda)$. For the $g_\eps$-length it holds $\ell_\eps(\gamma_\lambda) = \sqrt{2H_\eps(\lambda)}$. It follows that $U_\eps \subseteq \{\delta_\eps <a\}$. Furthermore, for any point  $x\in \{\delta_\eps <a\}$ there exists by compactness at least one geodesic such that its length coincides with the distance $\delta_\eps(x)$ of $x$ from $\partial \Omega$. Such a geodesic is necessarily normal by Proposition \ref{prop:no_abn}, its initial covector $\lambda$ must be in $\A(\partial \Omega)$ by \eqref{eq:trcondition}, and indeed $\sqrt{2H_\eps(\lambda)} <a$. There is a unique such a covector in $V_\eps$, and any other covector not in $V_\eps$ yields a longer geodesic. It follows that $U_\eps = \{\delta_\eps <a\}$, and furthermore
\begin{equation}
\delta_\eps(E_\eps(\lambda)) = \sqrt{2H_\eps(\lambda)}, \qquad \forall\, \lambda \in V_\eps.
\end{equation}
In particular $\delta_\eps :\bar\Omega \to \R$ is smooth on $U_\eps \cap \bar\Omega$, for all fixed $\eps \in I$, proving the claim (notice that $2H_\eps$ is a homogeneous norm of degree $2$, on a one-dimensional space, so that $\delta_\eps$ is smooth up to $\partial\Omega$).

Let now
\begin{equation}
U := \bigcap_{\eps \in I} U_\eps = \bigcap_{\eps \in I} \{\delta_\eps < a\} = \{\delta < a\},
\end{equation}
where in the last equality we used the monotonicity of $\delta_\eps$, which follows from the analogue property for the approximating distances $d_\eps$.  Notice also that $U\times I \subset F(V\times I)$, in particular $F^{-1}$ is a well-defined diffeomorphism on $U\times I$. Therefore
\begin{equation}
\delta_\eps(q) = \sqrt{2H_\eps \circ F^{-1}(q,\eps)}, \qquad \forall\, q \in U, \; \eps \in I.
\end{equation}
The above formula, together with the smoothness of $H_\eps$ and the fact that it is a well-defined quadratic form on the one-dimensional fibers $\A_q(\partial\Omega)$, implies the joint smoothness  of $(q,\eps)\mapsto \delta_\eps(q)$ in both variables and up to the boundary on $U\cap \bar\Omega \times I$.
\end{proof}

We equip $M$ with a smooth measure $\omega$, which will be the same for the \sr structure $(\D,g)$, and for the Riemannian variation $g_\eps$. Let $\Delta_\eps = \diverg \circ \nabla_\eps$ be the corresponding weighted Laplace-Beltrami on $(M,g_\eps)$, where $\nabla_\eps$, for all $\eps \neq 0$, is the Riemannian gradient for $g_\eps$. Notice that the semi-groups associated with the Dirichlet extensions of $\Delta$ and $\Delta_\eps$ are defined on the same $L^2=L^2(\Omega,\omega)$. Denote by $Q_\Omega^\eps$ the corresponding $\eps$-Riemannian heat content. Recall that, as explained in Section \ref{sec:riem_approx}, we have
\begin{equation}
Q_\Omega^\eps(t)\xrightarrow{\eps\to 0}Q_\Omega(t), \qquad \text{uniformly on } [0,T].
\end{equation}
We know that, for any $\eps \neq 0$, there exists a complete asymptotic series of $Q_\Omega^\eps(t)$, namely 
\begin{equation}
Q_\Omega^\eps(t)\sim\sum_{k=0}^\infty a_k^\eps t^{k/2},\qquad\text{as }t\to 0.
\end{equation}
Moreover, a recursive formula for the coefficients is provided in \cite{Savo-heat-cont-asymp} (notice that the results therein hold for the Riemannian measure, but as we have proven in the previous sections one can generalize these formulas for the arbitrary measure $\omega$). Define the operators $D_k^\eps$ as in \eqref{eqn:recurs_op1}--\eqref{eqn:recurs_op3}, replacing the sub-Laplacian $\Delta$ with the $\eps$-Riemannian one $\Delta_\eps$, and the operator $\normal$, with the corresponding Riemannian one $\normal_\eps$, defined by
\begin{equation}
\label{eqn:riem_normal}
N_\eps\phi= 2g_\eps\left(\nabla_\eps\phi,\nabla_\eps\delta_\eps\right)+\phi\Delta_\eps\delta_\eps,
\end{equation}
where $\delta_\eps\colon\bar\Omega\to[0,\infty)$ is the $\eps$-Riemannian distance from $\partial\Omega$. In particular, $D_k^\eps$ belongs the algebra generated by $\Delta_\eps$ and $\normal_\eps$. Then, as in the \sr case, the $k$-th coefficient is given by
\begin{equation}
\label{eqn:riem_coeff}
a_k^{\eps}=\int_{\partial\Omega}D_k^\eps(1)d\sigma_\eps,
\end{equation}
where $\sigma_\eps$ is the induced Riemannian measure on $\partial\Omega$.

\begin{thm}
\label{thm:conv_coeff}
Let $M$ be a sub-Riemannian manifold, equipped with a smooth measure $\omega$, and let $\Omega \subset M$ be an open relatively compact subset whose boundary is smooth and has not characteristic points. Then, there exists a family of Riemannian metrics $g_\varepsilon$ such that $d_{\varepsilon} \to d_{\mathrm{SR}}$ uniformly on compact sets of $M$, and such that
\begin{equation}\label{eqn:conv_coeff}
\lim_{\varepsilon \to 0} a_k^\varepsilon = a_k, \qquad \forall\, k \in \mathbb{N},
\end{equation}
where $a_k$ and $a_k^\varepsilon$ denote the coefficients of the sub-Riemannian small-time heat content asymptotics, and the corresponding ones for the Riemannian approximating structure.
\end{thm}

\begin{rmk}
We believe that Theorem \ref{thm:conv_coeff} can be strengthened and that the convergence result actually holds for \emph{any} one-parameter family of Riemannian structures $\{(M,g_\eps)\}$ such that the corresponding Hamiltonians $\{H_\eps\}$ converge to the sub-Riemannian one (uniformly on compact sets as $\eps \to 0$), cf.\ \eqref{eqn:hameps}. We refrained to prove this more general result as it would not add new ideas. Furthermore, one advantage of the current version of Theorem \ref{thm:conv_coeff} is that the construction of the approximating family is explicit.
\end{rmk}

\begin{proof}
First of all, we have that $\Delta_\eps\delta_\eps\to\deltasr\delta$ as $\eps\to 0$, uniformly on the compact sets of $U$, with all the derivatives, where $U$ is the set given by Lemma \ref{lem:smooth_dist}. Indeed, using the explicit expression of the $\eps$-Laplacian provided in \eqref{eqn:approx_lapl}, with respect to a global generating frame for the Riemannian variation, we have
\begin{equation}\label{eqn:conv_lapl0}
\Delta_\eps=\deltasr+\eps^2 \tilde{\Delta},
\end{equation}
where $\tilde{\Delta}$ is a second order differential operator independent from $\eps$. Thanks to Lemma \ref{lem:smooth_dist}, we know that $\tilde{\Delta}\delta_\eps\to \tilde{\Delta}\delta$ as $\eps\to 0$, with all the derivatives, uniformly on the compact subsets of $U$. Hence, we get
\begin{equation}
\label{eqn:conv_lapl2}
\Delta_\eps\delta_\eps=\deltasr\delta_\eps+\eps^2 \tilde{\Delta}\delta_\eps\xrightarrow{\eps\to 0}\deltasr\delta, 
\end{equation}
with all the derivatives, uniformly on the compact subsets of $U$. 

Second of all, we claim that
\begin{equation}
\label{eqn:claim_op}
D_k^\eps(\phi)\xrightarrow{\eps\to 0} D_k(\phi),\qquad \forall\, \phi \in C^\infty_c(U),
\end{equation}
uniformly on the compact subsets of $U$. The operators $D_k^\eps$ and $D_k$ are generated by finite combinations of elements in $\{\normal_\eps,\Delta_\eps\}$ and $\{\normal,\Delta\}$, respectively. Then, it is sufficient to prove that, for any sequence $\phi_\eps \in C_c^\infty(U)$ such that $\phi_\eps\to \phi$ uniformly with all the derivatives on the compact sets of $U$, and for any $s\in\N$, it holds
\begin{equation}
\label{eqn:to_prove}
\Delta_\eps^{s}\phi_\eps\xrightarrow{\eps\to 0}\deltasr^{s}\phi\qquad\text{and}\qquad N_\eps^s\phi_\eps\xrightarrow{\eps\to 0}\normal^s\phi,
\end{equation}
uniformly on the compact sets of $U$, with all the derivatives. The first statement in \eqref{eqn:to_prove} follows directly from \eqref{eqn:conv_lapl0}. To prove the second statement of \eqref{eqn:to_prove}, proceed by induction on $s\in\N$: for $s=1$, let us write explicitly
\begin{equation}
\label{eqn:induct_N1}
N_\eps\phi_\eps= 2g_\eps\left(\nabla_\eps\phi_\eps,\nabla_\eps\delta_\eps\right)+\phi_\eps\Delta_\eps\delta_\eps.
\end{equation}
From \eqref{eqn:conv_lapl2}, the second term in \eqref{eqn:induct_N1} converges as required. On the other hand, using the definition of $\eps$-gradient, we have
\begin{equation}
\label{eqn:expr_grad}
g_\eps\left(\nabla_\eps\phi_\eps,\nabla_\eps\delta_\eps\right)=d\delta_\eps(\nabla_\eps\phi_\eps).
\end{equation}
Now, thanks to \eqref{eqn:approx_grad}, we can write the $\eps$-Riemannian gradient with respect to the global generating frame \eqref{eqn:gen_frame3}, as
\begin{equation}
\nabla_\eps\phi_\eps=\grad\phi_\eps+\eps^2\tilde\nabla\phi_\eps,\qquad\forall\, \eps>0,
\end{equation}
where the sequence $\{\tilde\nabla\phi_\eps\}$ has coefficients which are uniformly bounded, together with all the derivatives, on the compact subsets of $U$. Thus, using \eqref{eqn:expr_grad}, we obtain
\begin{equation}
\label{eqn:prodeps_expr}
g_\eps\left(\nabla_\eps\phi_\eps,\nabla_\eps\delta_\eps\right)=d\delta_\eps(\grad\phi_\eps)+\eps^2d\delta_\eps (\tilde{\nabla}\phi_\eps)=g(\grad \delta_\eps,\grad\phi_\eps)+\eps^2d\delta_\eps (\tilde{\nabla}\phi_\eps).
\end{equation}
Since the sequence of smooth functions $\{d\delta_\eps (\tilde{\nabla}\phi_\eps)\}$ is uniformly bounded on the compact subsets of $U$, together with all the derivatives, the second term in \eqref{eqn:prodeps_expr} converges to $0$. For the first term in \eqref{eqn:prodeps_expr}, using Lemma \ref{lem:smooth_dist}, we have
\begin{equation}
g(\grad \delta_\eps,\grad\phi_\eps)\xrightarrow{\eps\to 0}g(\grad\phi,\grad\delta)
\end{equation}
with all the derivatives, uniformly on the compact sets of $U$, implying that \eqref{eqn:induct_N1} converges as required. Assume that \eqref{eqn:to_prove} holds for $s-1$, for any sequence $\{\phi_\eps\}$. Define $\psi_\eps= N_\eps^{s-1}\phi_\eps$ then, by induction hypothesis $\psi_\eps\to \normal^{s-1}\phi=\psi$, with all the derivatives, uniformly on the compact sets of $U$. Thus, applying the previous step, we have
\begin{equation}
N_\eps\psi_\eps\xrightarrow{\eps\to 0}\normal\psi,
\end{equation}
with all the derivatives, uniformly on the compact sets of $U$. This proves claim \eqref{eqn:claim_op}.

Finally, Lemma \ref{lem:smooth_dist} implies that the sequence of Riemannian measures induced on $\partial\Omega$, i.e.\ $\{\sigma_\eps\}$ converges weakly to $\sigma$ and, from the explicit expression of both the Riemannian and \sr coefficients, respectively \eqref{eqn:riem_coeff} and \eqref{eqn:full_asymp}, we conclude.
\end{proof} 		% convergence of the coefficients
\section{Blow-up of \texorpdfstring{$a_5$}{a5} in a domain with characteristic points}\label{s:blowup}
\label{sec:integrability}
We provide here an explicit example of a domain with an isolated characteristic point, in which the fifth coefficient of the asymptotic expansion \eqref{eqn:full_asymp} blows up, proving Theorem \ref{t:3intro} in the Introduction. Let us consider the \sr structure on $\R^3$, defined by the global generating frame
\begin{equation}
X_1=\partial_x-\frac{y}{2}\partial_z,\qquad X_2=\partial_y+\frac{x}{2}\partial_z,
\end{equation}
and we set $\{X_1,X_2\}$ to be an orthonormal frame. The resulting \sr manifold is the well-known first Heisenberg group, $\hei$. We equip it with the standard Lebesgue measure, and the corresponding sub-Laplacian is $\Delta=X_1^2+X_2^2$. Let $\Sigma$ to be the $xy$-plane, i.e. the zero-level set of the function
\begin{equation}\label{eqn:bdfun}
u(x,y,z)= z.
\end{equation}
The characteristic points are solution to $X_1(u)(x,y,z)=X_2(u)(x,y,z)=0$. Thus, $\Sigma$ has only one isolated characteristic point at $(0,0,0)$. The \sr distance from $\Sigma$, denoted by $\delta : \hei \to \R$, remains smooth in a neighborhood of non-characteristic points but, at the origin, it is no longer smooth. Thus, to investigate the behavior of the heat content coefficients, we need an explicit expression for $\delta$.

\begin{rmk}
The surface $\Sigma$ is not compact. Since we are interested in local integrability properties of the coefficient $a_5$ around a characteristic point, it is not restrictive to work with a non-compact surface.
\end{rmk}

In order to study the function $\delta$, we employ the symmetries of the Heisenberg group, to obtain a substantial dimensional reduction. First of all, the distance function $d(\cdot,\cdot)$ is $1$-homogeneous with respect to the one-parameter family of dilations
\begin{equation}
\label{eqn:dilation}
\chi_t:\hei\rightarrow\hei;\qquad\chi_t(x,y,z)=(tx,ty,t^2z),\qquad \forall t>0.
\end{equation} 
This means that
\begin{equation}
d(\chi_t(p),\chi_t(q))=td(p,q), \qquad\forall\,p,q\in\hei.
\end{equation}
Since the dilations $\chi_t$ are injective and fix $\Sigma$, we obtain the analogous property for $\delta$:
\begin{equation}
\delta(\chi_t(p))=\inf_{q\in\Sigma}d(\chi_t(p),q)=\inf_{\chi_t(q)\in\Sigma}d(\chi_t(p),\chi_t(q))=\inf_{\chi_t(q)\in\Sigma} td(p,q)=t\delta(p).
\end{equation}

Second of all, an isometry $L: \hei\rightarrow\hei$ preserves the distance from $\Sigma$ if and only if $L(\Sigma)=\Sigma$.  Among the isometries of $\hei$, the rotations around the $z$-axis and the reflection
\begin{equation}
L\colon\hei\rightarrow\hei, \qquad L(x,y,z)=(x,-y,-z)
\end{equation} 
preserve $\Sigma$. This means that, when expressed in cylindrical coordinates
\begin{equation}
\label{eqn:cil_coord}
\begin{cases}
x=r\cos\varphi\\
y=r\sin\varphi\\ 
z=z
\end{cases}
\end{equation}
the function $\delta$ does not depend on the angle $\varphi$, nor on the sign of $z$. 

Thus, letting $p=(r_p,\varphi_p,z_p)$ and assuming $r_p\neq 0$, we have
\begin{equation}
\label{eqn:symm_dist}
\delta(p)=\delta(r_p,\varphi_p,z_p)=\delta(r_p,0,|z_p|)=r_p\delta\left(1,0,\frac{|z_p|}{r_p^2}\right)= r_p F\left(\frac{|z_p|}{r_p^2}\right),
\end{equation}
for a suitable function $F$.

\subsection{Cut and focal locus from \texorpdfstring{$\Sigma$}{{z=0}}}

We refer to the preliminaries in Section \ref{sec:geod} for basic facts about geodesics in sub-Riemannian geometry. The singular points of $\delta$ are related to where the geodesics spreading out from $\Sigma$ lose optimality. 

Recall that $\hei$ has no non-trivial abnormal geodesics. Thus, by Proposition \ref{prop:no_abn}, we may focus on $\Sigma_0 = \Sigma \setminus \{0\}$. By the transversality condition \eqref{eq:trcondition}, for any point $(x_0,y_0,0) \in \Sigma_0$, there exists a unique normal geodesic parametrized by arc-length with initial covector
\begin{equation}
\lambda(0)=\frac{2}{\sqrt{x_0^2+y_0^2}}dz.
\end{equation}
Writing in coordinates $\lambda=(\lambda_x,\lambda_y,\lambda_z;x,y,z)$ the Hamiltonian system \eqref{eq:Hamiltoneqs}, we obtain
\begin{align}
\dot x & =\lambda_x-\frac{y}{2}\lambda_z, & \qquad  \dot{\lambda}_x & =-\frac{1}{2}\lambda_z\left(\lambda_y+\frac{x}{2}\lambda_z\right),\\ 
\dot y & =\lambda_y+\frac{x}{2}\lambda_z, &  \qquad   \dot{\lambda}_y & =\frac{1}{2}\lambda_z\left(\lambda_x-\frac{y}{2}\lambda_z\right),\\
\dot z & =\frac{1}{2}(x\dot y-\dot xy), &  \qquad   \dot{\lambda}_z & =0.
\end{align}
Imposing the initial condition $\lambda(0)=(0,0,\frac{2}{r_0};x_0,y_0,0)$, we can explicitly compute the geodesic starting from the point $(x_0,y_0,0)\in\Sigma_0$
\begin{equation}
\label{eqn:gamma}
\gamma_{x_0,y_0}(t)= \frac{1}{2}\begin{pmatrix} 
x_0\left(1+\cos\left(\frac{2t}{r_0}\right)\right)-y_0\sin\left(\frac{2t}{r_0}\right) \\ 
y_0\left(1+\cos\left(\frac{2t}{r_0}\right)\right)+x_0\sin\left(\frac{2t}{r_0}\right) \\
\frac{r_0^2}{4}\left(\frac{2t}{r_0}+\sin\left(\frac{2t}{r_0}\right)\right)
\end{pmatrix}.
\end{equation}
Define then the map $E: [0,\infty) \times \Sigma_0 \to \hei$, by
\begin{equation}
\label{eqn:exp_bd}
E(t;x_0,y_0)=\gamma_{x_0,y_0}(t).
\end{equation} 
We remark that \eqref{eqn:exp_bd} corresponds to the restriction of the \sr exponential map to the annihilator bundle of $\Sigma$, given by \eqref{eq:trcondition}. We define the \emph{focal locus} to be the set of critical values of $E$. 
\begin{lem}
\label{lem:focal_locus}
The focal locus of $\Sigma$ coincides with the $z$-axis, minus the origin.
\end{lem}

\begin{proof}
Representing the differential $dE$ in polar coordinates, we get that
\begin{equation}
\mathrm{det}(dE)=\frac{1}{4}\left(r_0\left(1+\cos\left(\frac{2t}{r_0}\right)\right)+t\sin\left(\frac{2t}{r_0}\right)\right),
\end{equation}
and we look for solutions to the equation $\mathrm{det}(dE)=0$. For any $(x_0,y_0) \in \Sigma_0$ the critical points are obtained at $t=t_c=r_0\left(\tfrac{\pi}{2}+k\pi\right)$, for $k\in \Z$, with critical values
\begin{equation}
E(t_c;x_0,y_0)=\left(0,0,\frac{r_0^2}{8}(\pi+2k\pi)\right).
\end{equation}
Thus, critical values are all points $(0,0,\xi)\in \hei$ with $\xi \neq 0$.
Notice that  for all such points there exists a one-parameter family of geodesics joining $\Sigma$ to that point.
\end{proof}

We define the \emph{cut locus of $\Sigma$} as the set of all those points in which the map $E$ fails to be injective, i.e. $p\in\cut{\Sigma}$ if and only if $E(t;x_0,y_0)=E(t;\bar x_0,\bar y_0)=p$, for some points $(x_0,y_0,0), (\bar x_0,\bar y_0,0)\in\Sigma$.

\begin{lem}
\label{lem:cut_locus}
For a point $p\in\hei$, not lying on the $z$-axis, there exists a unique geodesic which realizes the distance from $\Sigma$. Thus $p\notin\cut{\Sigma}$. 
\end{lem}

\begin{proof}
First of all, we notice that the cut locus is invariant under isometries which preserve the horizontal plane.

Second of all, proceed by contradiction and assume there exists a point $p\in\cut{\Sigma}\setminus\{z=0\}$. Then, for any $t\in\R$, $\chi_t(p)\in\cut{\Sigma}$, where $\chi_t$ is the dilation defined in \eqref{eqn:dilation}. In an analogous way, the rotations around the $z$-axis preserve the cut locus. Therefore, we get that if $p\in\cut{\Sigma}$, the cut locus contains a paraboloid with base point the origin, passing through $p$, see Figure \ref{fig:paraboloid}.
\begin{figure}[ht]
\centering
\includegraphics[height=6cm]{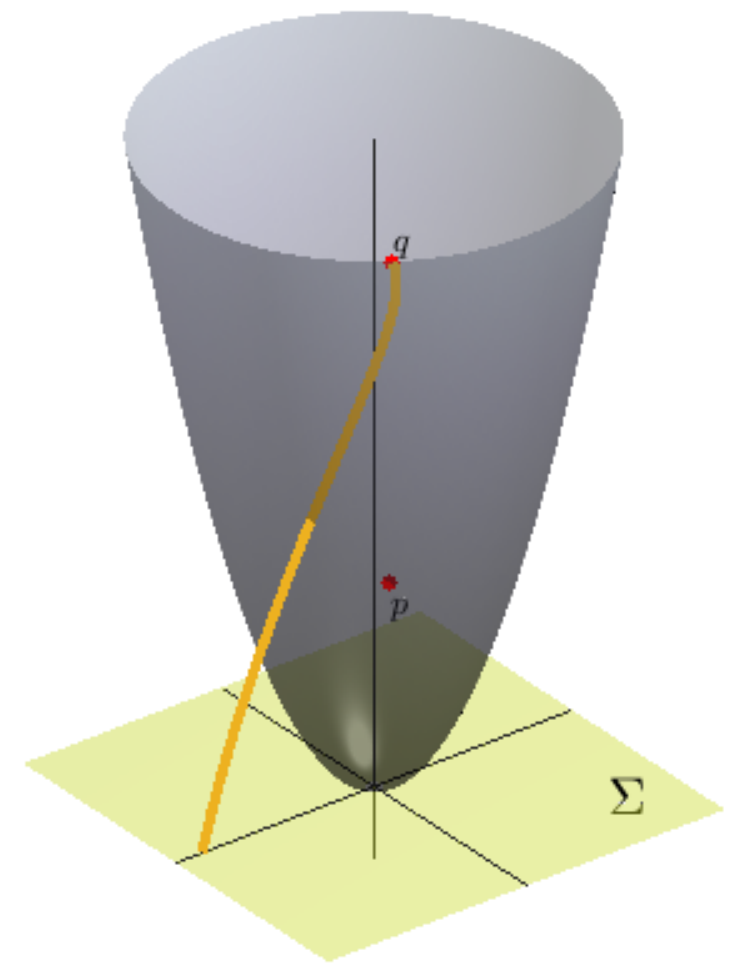}
\caption{The paraboloid passing through $p=(1,1,2)$ is contained in $\cut{\Sigma}$.}
\label{fig:paraboloid}
\end{figure}  
The paraboloid, contained in the half-space containing $p$, separates it in two connected components, $C_1$ to which $\Sigma$ belongs, and $C_2$. Now, pick any point $q\in C_2\setminus\cut{\Sigma}$: since $q$ is not in the cut locus, there exists a unique minimizing geodesic joining $q$ and $\Sigma$, however this geodesic must cross the paraboloid, hence the cut locus, losing minimality. This gives a contradiction. 

Notice that the existence of a point $q\in C_2\setminus\cut{\Sigma}$ is guaranteed by the fact that the cut locus is a nowhere dense set (see \cite{ACS_reg} and \cite{Agrasmoothness}, for further details).
\end{proof}

\begin{cor}
\label{lem:cut_zaxis}
The cut locus of $\Sigma$ coincides with $z$-axis, minus the origin. Moreover, for all $p$ on the $z$-axis, we have the following formula for the distance function from $\Sigma$
\begin{equation}
\delta(p)=\delta(0,0,z_p)=\sqrt{2\pi|z_p|}, \qquad \forall\,p\in z\text{-axis.}
\end{equation}
\end{cor}

\begin{proof}
By the last lines in the proof of Lemma \ref{lem:focal_locus}, we see that the $z$-axis (minus the origin) is contained in the cut locus of $\Sigma$ and, by Lemma \ref{lem:cut_locus}, we conclude that $\cut{\Sigma}=\{(0,0,z)\mid z \neq 0\}$. To explicitly compute $\delta(p)$ for a point $p=(0,0,z_p)$ in the $z$-axis, we look for the smallest positive time $t=\tmin$ for which $E(t;x_0,y_0)=p$, where $E$ is defined in \eqref{eqn:exp_bd}. We may assume without loss of generality that $z_p>0$. In cylindrical coordinates \eqref{eqn:cil_coord}, $E(\tmin;x_0,y_0)=p$ is given by
\begin{equation}
\label{eqn:endpoint_sys}
\begin{cases}
0=\left(\cos\varphi_0+\cos\left(\varphi_0+\frac{2\tmin}{r_0}\right)\right),\\
0=\left(\sin\varphi_0+\sin\left(\varphi_0+\frac{2\tmin}{r_0}\right)\right),\\
z_p=\frac{r_0^2}{8}\left(\frac{2\tmin}{r_0}+\sin\left(\frac{2\tmin}{r_0}\right)\right).
\end{cases}
\end{equation}
Using the first two equations of \eqref{eqn:endpoint_sys}, we see that
\begin{equation}
\varphi_0+\frac{2\tmin}{r_0}=\varphi_0+\pi+2k\pi\qquad\Rightarrow\qquad\frac{2\tmin}{r_0}=\pi+2k\pi,
\end{equation}
with $k$ a positive integer. In particular, since $\tmin$ has to be minimal, $k=0$, therefore $\tmin=\frac{r_0\pi}{2}$ and $r_0$ has to satisfy the third equation of \eqref{eqn:endpoint_sys}
\begin{equation}
z_p=\frac{r_0^2\pi}{8}.
\end{equation}
Notice that the above equation uniquely determines $r_0$, while $\varphi_0$ can vary in the interval $(0,2\pi)$. Thus $p\in\cut{\Sigma}$ and 
\begin{equation}
\delta(p)=\tmin=\frac{r_0\pi}{2}=\sqrt{2z_p\pi},
\end{equation}
concluding the proof.
\end{proof}

\subsection{An expression for the distance function from \texorpdfstring{$\Sigma$}{{z=0}}}
\label{ssec:expr_distance}
We are interested in an explicit expression of the distance function outside the $z$-axis, where we expect it to be smooth.

Assume that $p\in\hei$ is not on the $z$-axis. This implies that $r_p\neq 0$ and, exploiting the symmetries of the distance from $\Sigma$ as in \eqref{eqn:symm_dist}, we may assume that $p=(1,0,\xi)$, where $\xi>0$. Therefore, we rewrite the system \eqref{eqn:endpoint_sys} for $p=(1,0,\xi)$:
\begin{equation}
\label{eqn:endpoint_sys2}
\begin{cases}
1=\frac{r_0}{2}\left(\cos\varphi_0+\cos\left(\varphi_0+\frac{2\tmin}{r_0}\right)\right),\\
0=r_0\left(\sin\varphi_0+\sin\left(\varphi_0+\frac{2\tmin}{r_0}\right)\right),\\
\xi=\frac{r_0^2}{8}\left(\frac{2\tmin}{r_0}+\sin\left(\frac{2\tmin}{r_0}\right)\right),
\end{cases}
\end{equation}
using again the cylindrical coordinates \eqref{eqn:cil_coord}, for the point $(x_0,y_0)$.
Solving the second equation of the system, we get two possible solutions:
\begin{equation}
\varphi_0+\frac{2\tmin}{r_0}=\varphi_0+\pi+2k\pi, \qquad \varphi_0+\frac{2\tmin}{r_0}=-\varphi_0+2k\pi.
\end{equation}
In the first case, we don't have any solution of the system \eqref{eqn:endpoint_sys2}, since its first equation is not verified. In the second case, we obtain
\begin{equation}
\label{eqn:endpoint_sys3}
\begin{cases}
1=\frac{r_0}{2}\left(\cos\varphi_0+\cos\left(\varphi_0+\frac{2\tmin}{r_0}\right)\right)=r_0\cos\varphi_0,\\
\tmin=r_0(k\pi-\varphi_0),
\end{cases}
\end{equation}
where $k$ is an integer parameter to be fixed. Replacing the expression for $\tmin$ in the third equation of \eqref{eqn:endpoint_sys2}, we obtain the following equality for $k$:
\begin{equation}
\xi=\frac{r_0^2}{8}\left(2k\pi-2\varphi_0+\sin\left(-2\varphi_0\right)\right)=\frac{r_0}{4}\left(r_0(k\pi-\varphi_0)-\sin(\varphi_0)\right).
\end{equation}
Thus, the integer $k$, as a function of the base point $(1,y_0)$ and the final point $(1,0,\xi)$ of the geodesic $\gamma$, is given by
\begin{equation}
k(\xi,y_0)=\frac{4\xi+r_0\sin(\varphi_0)+r_0^2\varphi_0}{r_0^2\pi}=\frac{4\xi+y_0+(1+y_0^2)\arctan(y_0)}{(1+y_0^2)\pi}.
\end{equation} 
Here we are using the first equation of \eqref{eqn:endpoint_sys2}, which tells us that $r_0\cos(\varphi_0)=1$. Furthermore, using \eqref{eqn:endpoint_sys3}, we can compute the length of the corresponding geodesic $\gamma$ as a function of $\xi$ and $y_0$:
\begin{equation}
t(\xi,y_0)= r_0\left(k(\xi,y_0)\pi-\varphi_0\right)=\frac{4\xi+y_0}{\sqrt{1+y_0^2}}.
\end{equation}
 Finally, to compute $\delta(p)$, we have to find $y_0$, realizing the minimum
\begin{equation}
\tmin=\min\left\{t(\xi,y) \mid y\in\R,\quad k(\xi,y)\in\Z\right\}.
\end{equation}
We will refer to those points for which $k(\xi,y)\in\Z$, as admissible starting points.

For fixed $\xi>0$, we plotted in Figures \ref{fig:timefun} and \ref{fig:kfun} the graphs of the functions
\begin{equation}
y\mapsto t(\xi,y),\qquad y\mapsto k(\xi,y).
\end{equation}
\begin{figure}[t]
\centering
\subfigure[{The function $t(\xi,\cdot)$}\label{fig:timefun}]
{\includegraphics[height=5cm]{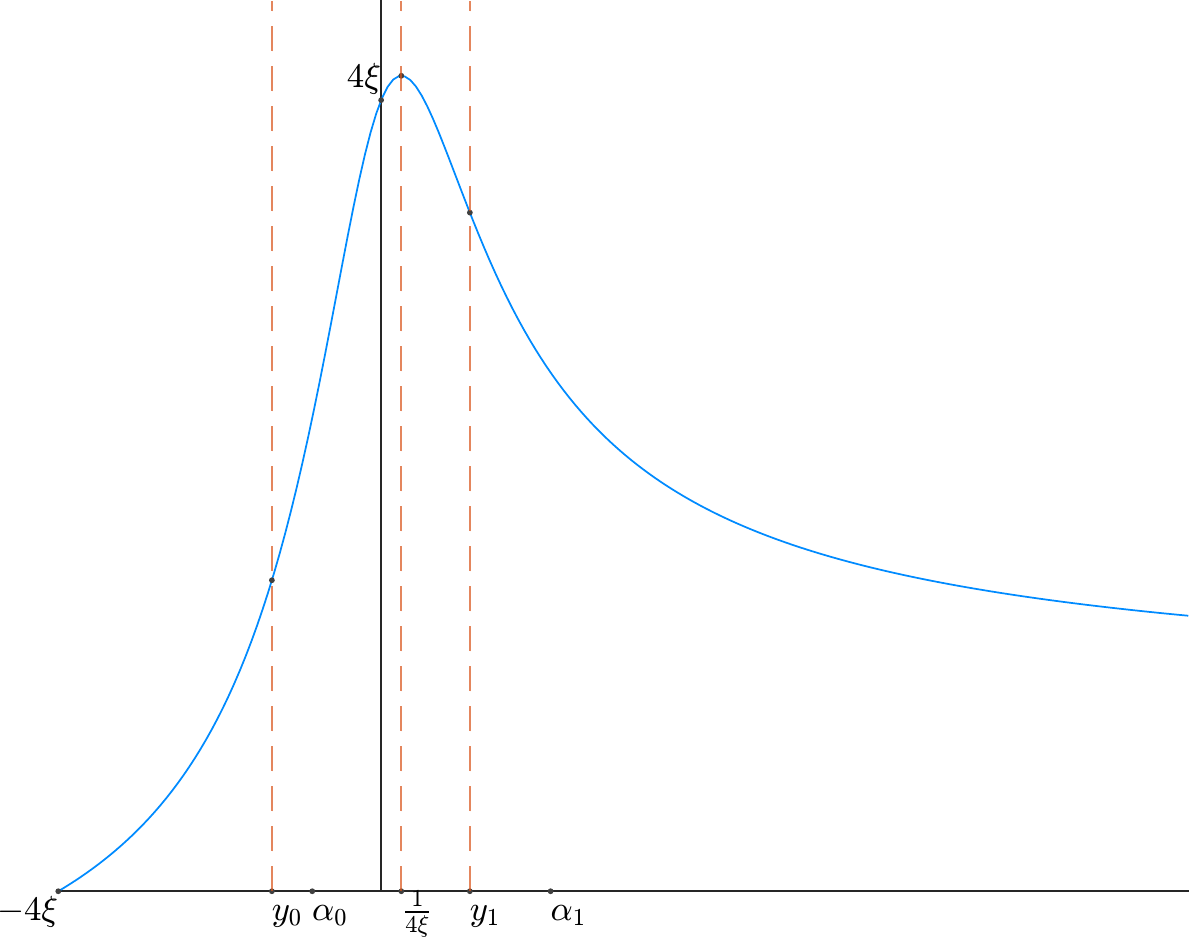}}
\hfill
\subfigure[{The function $k(\xi,\cdot)$}\label{fig:kfun}]
{\includegraphics[height=5cm]{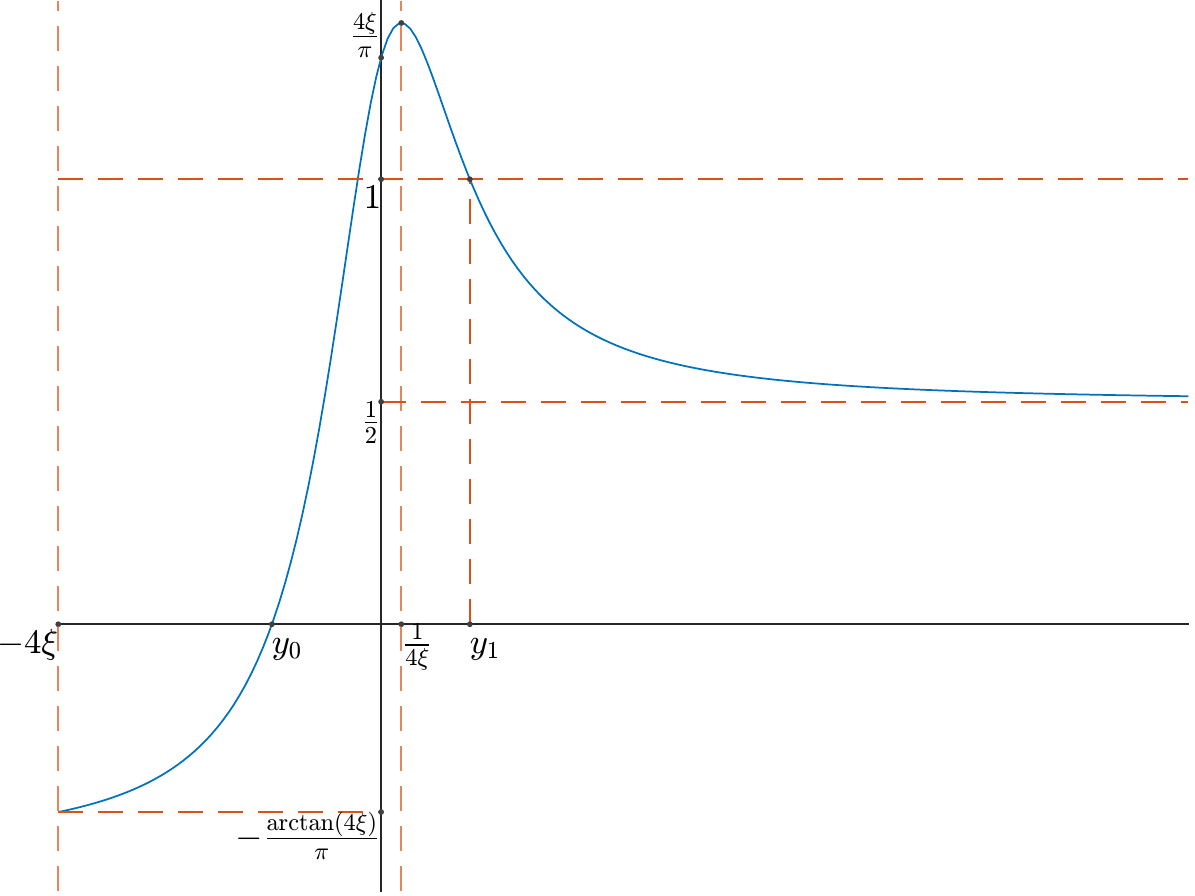}}
\caption{Functions occurring in the determination of the distance from $\Sigma$.}
\end{figure}
First of all, we notice that the admissible starting points must satisfy $y\geq -4\xi$. Second of all, we see that the equation $k(\xi,y)=0$ has a unique solution for every value of $\xi$, and it lies in the interval $(-4\xi,0)$, being $k(\xi,0)>0$ and $k(\xi,-4\xi)<0$. Moreover, since both these functions have a unique maximum at the point $\frac{1}{4\xi}$, they are both increasing in the interval $[-4\xi,\frac{1}{4\xi})$ and decreasing in the interval $[\frac{1}{4\xi},+\infty)$. Thus, among the admissible starting points in $[-4\xi,\frac{1}{4\xi})$, the minimum time (for these starting points) is achieved at the point for which $k(\xi,\cdot)$ is the minimum integer. Therefore, since
\begin{equation}
k(\xi,y)\geq k(\xi,-4\xi)=-\frac{\arctan(4\xi)}{\pi}>-\frac{1}{2},\qquad\forall\,y\geq-4\xi,
\end{equation}
we have that in the interval $[-4\xi,\frac{1}{4\xi})$, the minimum time is achieved at $y_0$, such that $k(\xi,y_0)=0$. On the other hand, in an analogous way, since 
\begin{equation}
k(\xi,y)\geq \lim_{y\to\infty}k(\xi,y)=\frac{1}{2},\qquad\forall\,y\geq\frac{1}{4\xi},
\end{equation}
in the interval $[\frac{1}{4\xi},+\infty)$, the minimum time is achieved at $y_1$, such that $k(\xi,y_1)=1$, if exists. Hence, we are left to compare, for $i=0,1$
\begin{equation} 
t(\xi,y_i)=\frac{4\xi+y_i}{\sqrt{1+y_i^2}},\qquad\text{with }k(\xi,y_i)=i.
\end{equation}
The existence of $y_1$, such that $k(\xi,y_1)=1$, is guaranteed if the maximum of $k(\xi,\cdot)$ is $\geq 1$. Thus, if this is not the case, i.e. if 
\begin{equation}
\label{eqn:ineq_max}
k\left(\xi,\frac{1}{4\xi}\right)=\frac{4\xi+\arctan(\frac{1}{4\xi})}{\pi}< 1,
\end{equation}
we can immediately conclude that
\begin{equation}
\tmin=\frac{4\xi+y_0}{\sqrt{1+y_0^2}},\qquad\text{with }k(\xi,y_0)=0.
\end{equation}
Choosing $\xi\leq\frac{\pi}{8}$, the above inequality \eqref{eqn:ineq_max} is satisfied and we conclude.

Consider the case $\xi>\frac{\pi}{8}$. Let us write explicitly the equation satisfied by $y_0,y_1$ respectively, defining
\begin{align}
f_0(y)& = 4\xi+y+(1+y^2)\arctan(y),\\
f_1(y)& = (1+y^2)\pi-4\xi-y-(1+y^2)\arctan(y).
\end{align}
In this way, we will have $f_i(y_i)=0$. Recall that, since we are assuming $\xi>0$, we have $y_0<0$, and $y_1\in[\frac{1}{4\xi},+\infty)$, hence, we are looking for a negative zero of $f_0$ and a positive zero of $f_1$. Notice that
\begin{equation}
-\frac{\pi}{2}\leq\arctan(y)\leq 0, \qquad\forall\,y\leq 0,
\end{equation}
therefore, we have the following estimate for $f_0$:
\begin{equation}
\label{eqn:est_f_0}
-\frac{\pi}{2}y^2+y+4\xi-\frac{\pi}{2}\leq f_0(y),\qquad\forall\,y\leq 0.
\end{equation}
Thus, replacing $y_0$ in \eqref{eqn:est_f_0}, we are able to compare it with the zero of the parabola
\begin{equation}
y_0\leq\frac{1-\sqrt{1-\pi(\pi-8\xi)}}{\pi}=\alpha_0<0.
\end{equation}
We remark that we are in the case $\xi>\frac{\pi}{8}$, so $\alpha_0$ is always well-defined, moreover, since $y_0\leq\alpha_0<0$ and $t(\xi,\cdot)$ is increasing for negative arguments, $t(\xi,y_0)\leq t(\xi,\alpha_0)$, as one can check computing its derivative. See figure \ref{fig:timefun}. Reasoning in an analogous way for $f_1$, we obtain
\begin{equation}
\label{eqn:est_f_1}
\frac{\pi}{2}y^2-y+\frac{\pi}{2}-4\xi\leq f_1(y),\qquad\forall\,y\geq 0,
\end{equation}
which implies, for $y_1$, the following inequality:
\begin{equation}
0<y_1\leq\frac{1+\sqrt{1-\pi(\pi-8\xi)}}{\pi}=\alpha_1,
\end{equation}
and, as before $\alpha_1$ is always well-defined, if $\xi>\frac{\pi}{8}$. We remark that $\alpha_0,\alpha_1$ are the two zeroes of the parabola appearing in the left-hand sides of the estimates \eqref{eqn:est_f_0}, \eqref{eqn:est_f_1}. Since $0<y_1\leq\alpha_1$ and $t(\xi,\cdot)$ is decreasing for positive arguments, we have $t(\xi,\alpha_1)\leq t(\xi,y_1)$, see figure \ref{fig:timefun}.

Finally, to assert which point realizes the minimum time between $y_0, y_1$, it's enough to compare $t(\xi,\alpha_0),t(\xi,\alpha_1)$. One can check that, for any $\xi>\frac{\pi}{8}$
\begin{equation}
t(\xi,\alpha_0)<t(\xi,\alpha_1)\quad\Rightarrow\quad t(\xi,y_0)\leq t(\xi,\alpha_0)<t(\xi,\alpha_1)\leq t(\xi,y_1).
\end{equation}
This proves that, for points of the form $p=(1,0,\xi)$,
\begin{equation}
\label{eqn:function_F}
\delta(p)=\frac{4\xi+y_0}{\sqrt{1+y_0^2}}=F(\xi),\qquad\text{where }\ k(\xi,y_0)=0,
\end{equation}
where $F$ has been defined in \eqref{eqn:symm_dist}. Notice that this also provides an analytic proof of the fact that $p\notin\cut{\Sigma}$. Putting together the results of this section and Corollary \ref{lem:cut_zaxis}, we prove Theorem \ref{t:explicitformulaintro}.

\subsection{Blow-up of \texorpdfstring{$a_5$}{a5} and proof of Theorem \ref{t:3intro}}
Recall that for a domain $\Omega\subset M$, whose  boundary has no characteristic points, the fifth coefficient in the asymptotic expansion of the heat content is defined by the formula
\begin{equation}
\label{eqn:coeff_a5}
a_5=\frac{1}{240\sqrt{\pi}}\int_{\partial \Omega}(N^3-8N\deltasr)\deltasr \delta d\sigma,
\end{equation}
where $d\sigma$ is the \sr measure, induced by a fixed volume on $M$, and $N$ is the operator defined in \eqref{eqn:sr_normal}. We are going to show that, for the horizontal plane in the Heisenberg group, the integrand in \eqref{eqn:coeff_a5} is not locally integrable around the origin, which proves Theorem \ref{t:3intro}.

\begin{lem}
Let $\Sigma=\{z=0\} \subset \hei$. Then, the integrand in \eqref{eqn:coeff_a5} is
\begin{equation}
\label{eqn:integrand_a5}
\left(\frac{73}{640\sqrt{\pi}}\right)\frac{1}{r^4},
\end{equation}
whereas the \sr measure $d\sigma=r^2drd\varphi$, in cylindrical coordinates on $\hei$. In particular the integrand of the coefficient $a_5$ is not locally integrable with respect to the sub-Riemannian measure on $\Sigma$, around the characteristic point.
\end{lem}

\begin{rmk}
We explain the heuristic behind the order $r^{-4}$  in \eqref{eqn:integrand_a5}. In cylindrical coordinates, computing explicitly $\deltasr\delta$, we have $\deltasr\delta|_{\Sigma_0}=O\left(\tfrac{1}{r}\right)$. Moreover, the operator  $(N^3-8N\deltasr)$ has order $3$, and it contains a third-order derivation in the $r$-direction. 
\end{rmk}

\begin{proof}
By the symmetries of $\hei$, we can write $\delta$ as in \eqref{eqn:symm_dist}:
\begin{equation}
\delta(p)=rF\left(\frac{|z|}{r^2}\right),
\end{equation}
where $(r,\varphi,z)$ are the cylindrical coordinates for $p$ and $F$ has been explicitly computed in Section \ref{ssec:expr_distance}. In particular, by equation \eqref{eqn:function_F}
\begin{equation}
\label{eqn:function_F2}
F(\xi)=\frac{4\xi+y(\xi)}{\sqrt{1+y(\xi)^2}},
\end{equation}
and the function $y(\xi)$ satisfies
\begin{equation}
\label{eqn:function_y}
4\xi+y(\xi)+(1+y(\xi)^2)\arctan(y(\xi))=0.
\end{equation}
We can now compute explicitly the integrand of $a_5$. Rewriting the operators $N$ and $\deltasr$ in cylindrical coordinates, we obtain 
\begin{equation}
(N^3-8N\deltasr)\deltasr\delta|_{\Sigma_0}=\frac{C}{r^4},
\end{equation} 
where $C$ is a constant depending only on the derivatives of $F$ computed at the origin, up to order five. Hence, to obtain the integrand of $a_5$, it is enough to compute $F^{(k)}(0)$, for $k\leq 5$, which, by \eqref{eqn:function_F2}, amounts to compute $y^{(k)}(0)$, for $k\leq 5$. But this can be done using equation \eqref{eqn:function_y} and iterating the formula for the derivative of the implicit function. We omit the long but routine computation.
\end{proof}

\begin{rmk}\label{rmk:a4}
The integrand of $a_3$ is locally integrable on $\Sigma$ with respect to $d\sigma$, around the origin. Indeed, it is given by
\begin{equation}
-\left(\frac{3}{8\sqrt{\pi}}\right)\frac{1}{r^2},
\end{equation}
and $d\sigma=r^2drd\varphi$. On the other hand, heuristically, each subsequent coefficient $a_k$ requires an extra derivative of $1/r$. We may expect the first non-integrable coefficient to be $a_4$, close to a characteristic point. However, for the case of the $xy$-plane in $\hei$, one can check that the integrand of $a_4$ vanishes. Indeed, in this case, the function $F$ defining the distance, as in \eqref{eqn:function_F2}, is an odd function, therefore the integrands of all coefficients $a_{2i}$, which involve only even-order derivatives, vanish.
\end{rmk}		% blow-up of a_5 in the hor plane

\appendix
%\addappheadtotoc
%\appendixpage

%\input{convergence.tex}			% convergence of the heat semigroup of a Riemannian approximant
\section{Asymptotic expansion of \texorpdfstring{$I\phi(t,0)$}{Iphi(t,0)}}\label{app:B}

In this appendix, for self-consistence, we show how to prove Theorem \ref{thm:asymp_Iphi}, following closely \cite{Savo-heat-cont-asymp}. The proof consists in the application of an iterated version of the Duhamel's formula, that is an higher-order version of Lemma \ref{lem:duham_prin}.

\begin{prop}
\label{prop:iter_duhamel}
Let $F\in C^\infty((0,\infty)\times[0,\infty))$ be a smooth function compactly supported in the second variable. Assume that the following conditions hold:
\begin{itemize}
\item[(i)] $\displaystyle L^kF(0,r)=\lim_{t\to 0}L^kF(t,r)$ exists in the sense of distributions\footnote{Namely, for any $\psi\in C^\infty([0,\infty))$, there exists finite $\lim_{t\to 0}\int_0^\infty f(t,r)\psi(r)dr$. With a slight abuse of notation, we define $\int_0^\infty f(0,r)\psi(r)dr=\lim_{t\to 0}\int_0^\infty f(t,r)\psi(r)dr$.} for any $k\geq 0$; 
\item[(ii)] $L^kF(t,0)$ and $\partial_rL^kF(t,0)$ converge to a finite limit as $t\to 0$, for any $k\geq 0$.
\end{itemize}
Then, for all $m\in\N$ and $t>0$, we have 
\begin{multline}
\label{eqn:iter_duhamel}
F(t,0)=\sum_{k=0}^m{\left(\frac{t^k}{k!}\int_0^\infty{e(t,r,0)L^kF(0,r)dr}-\frac{1}{\sqrt{\pi}k!}\int_0^t{\partial_rL^kF(\tau,0)(t-\tau)^{k-1/2}d\tau}\right)}\\+\frac{1}{m!}\int_0^t{\int_0^\infty{e(t-\tau,r,0)L^{m+1}F(\tau,r)(t-\tau)^mdr} d\tau},
\end{multline}
where $e(t,r,s)$ is the Neumann heat kernel on the half-line, cf.\ \eqref{eqn:neumanheat}.
\end{prop} 
\begin{proof}
Define, for all $\epsilon>0$, the function $v(t,r) = F(t+\epsilon,r)$, for $t >0$ and $r\geq 0$. We stress that $v(t,r)$ is smooth on the closed set $[0,\infty)\times [0,\infty)$. In particular all the assumptions of Lemma \ref{lem:duham_prin} are verified with for the function $v$, and it holds
\begin{multline}
F(t+\epsilon,s) = \int_0^\infty e(t,s,r)F(\epsilon,r)dr -\int_0^t e(t-\tau,0,s)\partial_r F(\tau+\epsilon,0)d\tau \\
+ \int_0^t\int_0^\infty e(t-\tau,s,r)LF(\tau+\epsilon,r)drd\tau,
\end{multline}
for all $s\geq 0$ and $t>0$. Iterating the above formula, and using the semi-group property, we obtain for all $m\in \mathbb{N}$ and $\epsilon>0$
\begin{multline}
F(t+\epsilon,s)=\sum_{k=0}^m \left(\frac{t^k}{k!}\int_0^\infty{e(t,s,r)L^kF(\epsilon,r)dr}\right. \\
\left. -\frac{1}{k!}\int_0^t{e(t-\tau,0,s)\partial_rL^kF(\tau+\epsilon,0)(t-\tau)^{k}d\tau}\right)\\+\frac{1}{m!}\int_0^t{\int_0^\infty{e(t-\tau,r,s)L^{m+1}F(\tau+\epsilon,r)(t-\tau)^m dr} d\tau}.
\end{multline}
Then we send $\epsilon\to 0$ checking that, under the assumptions $(i)$ and $(ii)$, all terms on the right hand side is well-defined. We set finally $s=0$.
\end{proof}
Since we want to apply Proposition \ref{prop:iter_duhamel} to the function $I\phi(t,0)$, first, we study in detail the operators $L^kI$, for any $k\geq 1$. Recall the definition of the family of operators $R_{kj}$, $S_{kj}$ in \eqref{eqn:1st_fam1}--\eqref{eqn:1st_fam2}:
\begin{equation}
R_{00} = \mathrm{Id}, \qquad S_{00} = 0,
\end{equation}
while, for all $k\geq 1$, and $0\leq j \leq k$, define
\begin{align}
R_{kj}& =-(\normal^2-\deltasr)R_{k-1,j}+\normal S_{k-1,j}, \label{eqn:1st_fam1-rep}\\
S_{kj}& =\normal R_{k-1,j-1}-\deltasr\normal R_{k-1,j}+\deltasr S_{k-1,j}, \label{eqn:1st_fam2-rep}
\end{align}
and $R_{kj}=S_{kj}=0$, for all other values of the indices, i.e.\ $k<0$, $j<0$ or $k<j$. In addition, we define a second family, still lying in the algebra of operators generated by $\normal$ and $\Delta$. We set first
\begin{equation}
P_{00}=0, \qquad Q_{00}=\mathrm{Id},
\end{equation} 
while, for all $k\geq 1$, and $0\leq j \leq k$, we define
\begin{align}
P_{kj}& =-(\normal^2-\deltasr)P_{k-1,j}+\normal Q_{k-1,j}, \label{eqn:2nd_fam1} \\
Q_{kj}& =\normal P_{k-1,j-1}-\deltasr\normal P_{k-1,j}+\deltasr Q_{k-1,j}, \label{eqn:2nd_fam2}
\end{align}
and $P_{kj}=Q_{kj}=0$, for all other values of the indices, i.e.\ $k<0$, $j<0$ or $k<j$. These operators, which may seem obscure, arise naturally in the iterative application of the one-dimensional heat operator $L$ to $I\phi$ and $\Lambda\phi$, as we show in the next lemma.
Recall that $\omprime{\smallpar}=\bar\Omega\setminus\om{\smallpar}=\{x\in\bar\Omega\mid \delta(x)\leq\smallpar\}$.
\begin{lem}
\label{lem:iter_L}
Let $M$ be a sub-Riemannian manifold, equipped with a smooth measure $\omega$, and let $\Omega \subset M$ be an open relatively compact subset whose boundary is smooth and has not characteristic points. Denote by $\delta:\bar\Omega\to[0,\infty)$ the \sr distance function from $\partial\Omega$. Then, as operators on $\funspace$, we have:
\begin{itemize}
\item[(i)] $LI=\Lambda \normal+I\deltasr$;
\item[(ii)] $L\Lambda=\Lambda\left(-\normal^2+\deltasr\right)+\partial_tI\normal-I\deltasr \normal$;
\item[(iii)] For any $k\in\N$, 
\begin{equation}
\label{eqn:iter_expr}
L^kI=\sum_{j=0}^{k}{\frac{\partial^j}{\partial t^j}(\Lambda P_{kj}+IQ_{kj})}\qquad\text{and}\qquad L^k\Lambda=\sum_{j=0}^{k}{\frac{\partial^j}{\partial t^j}(\Lambda R_{kj}+IS_{kj})}.
\end{equation}
\end{itemize}
\end{lem}

\begin{rmk}
The time derivatives make sense, since, for any $\phi\in \funspace$, the functions $I\phi$, $\Lambda\phi$ are defined on $(t,r)\in (0,\infty)\times [0,\smallpar)$. Notice that only the operators $\Lambda$ and $I$ are time-dependent, while the families $P,Q,R,S$ do not depend on time.
\end{rmk}

\begin{proof}
Recall that $L=\partial_t-\partial_r^2$, hence, we use formula \eqref{eqn:2der_gap} to compute
\begin{equation}
\label{eqn:2der_I}
\partial_r^2I\phi(t,r)=\int_{\om{r}}{\deltasr\big((1-u(t,x))\phi(x)\big)\dvol(x)}-\int_{\bdom{r}}{(1-u(t,y))\phi(y)\deltasr\delta(y)\area}.
\end{equation}
Now, $\deltasr\big((1-u)\phi\big)=-\phi\deltasr u-2g(\grad u,\grad\phi)+(1-u)\deltasr\phi$, therefore the first term in \eqref{eqn:2der_I} becomes
\begin{align}
\int_{\om{r}}\deltasr((1-u)\phi)d\omega &=\int_{\om{r}}{\left(-\phi\deltasr u-2g(\grad u,\grad\phi)+(1-u)\deltasr\phi\right)d\omega}\\
																	&=-\int_{\om{r}}{\left(\phi\deltasr u+(1-u)\deltasr\phi\right)d\omega}+2\int_{\bdom{r}}{(1-u)g(\grad \phi,\nu)d\sigma},
\end{align}
where $\nu=-\grad\delta$ is the outward-pointing unit normal and where we used the divergence formula \eqref{eqn:diverg_thm}. Finally, \eqref{eqn:2der_I} becomes
\begin{align}
-\partial_r^2I\phi &=\int_{\om{r}}{\left(\phi\deltasr u+(1-u)\deltasr\phi\right)\dvol}+\int_{\bdom{r}}{(1-u)\left(-2g(\grad \phi,\nu)+\phi\deltasr\delta\right)d\sigma}\nonumber\\
									&=\int_{\om{r}}\phi\,\partial_tu\,\dvol+I\deltasr\phi+\Lambda \normal\phi=-\int_{\om{r}}{\partial_t(1-u)\phi\dvol}+I\deltasr\phi+\Lambda \normal\phi\nonumber\\
									&=-\partial_tI\phi+I\deltasr\phi+\Lambda \normal\phi.\label{eqn:2der_I_bis}
\end{align}
This concludes the proof of $(i)$, recalling that $L=\partial_t -\partial_r^2$.

To prove $(ii)$, we need to compute $L\Lambda\phi$. Since by definition $\partial_rI=-\Lambda$ (cf.\ Definition \ref{def:op_ILambda}), we rewrite the equality \eqref{eqn:2der_I_bis} as
\begin{equation}
\label{eqn:1der_Lambda}
\partial_r\Lambda\phi=-\partial_tI\phi+I\deltasr\phi+\Lambda \normal\phi,
\end{equation}
and we differentiate it with respect to $r$, obtaining
\begin{align}
\partial_r^2\Lambda\phi(t,r) &= \partial_t\Lambda\phi(t,r)-\Lambda\deltasr\phi(t,r)+\partial_r\Lambda \normal\phi(t,r)\\
														 &= \partial_t\Lambda\phi(t,r)-\Lambda\deltasr\phi(t,r)+\left(-\partial_tI+I\deltasr+\Lambda \normal\right)\normal\phi(t,r),
\end{align}
yielding the statement of $(ii)$. Point $(iii)$ follow easily by induction on $k$. 
\end{proof}

As suggested by Theorem \ref{thm:1ord_Q}, we would like to apply Proposition \ref{prop:iter_duhamel} to $I\phi(t,r)$ for $k\geq 2$, in order to obtain higher-order asymptotics. However, by Lemma \ref{lem:iter_L}, the terms $L^kI\phi$ involve time derivatives of $u(t,x)$ and these are not well-defined due to the lack of smoothness of $u$ at the boundary, at $t=0$. Therefore, we consider the following approximations of $I\phi$ and $\Lambda\phi$:
\begin{align}
I_\epsilon\phi(t,r) &= \int_{\om{r}}{(1-u_\epsilon(t,x))\phi(x)\dvol(x)}, \label{eqn:aux_ops1}
\\
\Lambda_\epsilon\phi(t,r) &= -\partial_r I_\epsilon\phi(t,r)= \int_{\bdom{r}}{(1-u_\epsilon(t,x))\phi(x)\dvol(x)}, \label{eqn:aux_ops2}
\end{align} 
where $u_\epsilon(t,x)$ denotes the solution to \eqref{eqn:dir_prob} with initial datum $\varphi(x)= \mathds{1}_{\om{\epsilon}}(x)$, where, we recall $\om{\epsilon} = \{x\in \Omega\mid \delta(x) >\epsilon\}$. Notice that, by the dominated convergence theorem, we have
\begin{equation}
I_\epsilon\phi(t,0)\xrightarrow{\epsilon\to 0}I \phi(t,0), 
\end{equation} 
uniformly on $[0,T]$. Moreover, Lemma \ref{lem:iter_L} holds unchanged for $I_\epsilon$ and $\Lambda_\epsilon$, and both $I_\epsilon\phi(t,r)$ and $\Lambda_\epsilon\phi(t,r)$ are compactly supported in the $r$-variable.

\begin{lem}
\label{lem:hp_idp_aux}
Under the same hypotheses of Lemma \ref{lem:iter_L}, let $\psi\in C^\infty([0,\infty))$, $\epsilon\in (0,\smallpar)$ and define
\begin{equation}
\psi^{(-1)}(r)=\int_0^r\psi(s)ds.
\end{equation}
Then, for any $\phi\in\funspace$, the following identities hold:
\begin{itemize}
\item[(i)] $\displaystyle\lim_{t\to 0}\int_0^\infty{\frac{\partial^j}{\partial t^j}\Lambda_\epsilon\phi(t,r) \psi(r)dr}=\begin{cases}\displaystyle\int_{\Omega\setminus\om{\epsilon}}{\phi(\psi\circ\delta)\dvol} & \text{ if }j=0,\\[10pt]
\displaystyle(-1)^j\int_{\om{\epsilon}}{\deltasr^j(\phi(\psi\circ\delta))\dvol}& \text{ if }j\geq 1;\end{cases}$
\item[(ii)] $\displaystyle\lim_{t\to 0}\int_0^\infty{\frac{\partial^j}{\partial t^j}I_\epsilon\phi(t,r)\psi(r)dr}=\begin{cases}\displaystyle\int_{\Omega\setminus\om{\epsilon}}{\phi\left(\psi^{(-1)}\circ\delta\right) \dvol}& \text{ if }j=0,\\[10pt]
\displaystyle(-1)^j\int_{\om{\epsilon}}{\deltasr^j\left(\phi\left(\psi^{(-1)}\circ\delta\right)\right) \dvol}& \text{ if }j\geq 1;\end{cases}$ 
\item[(iii)] $\forall\, t\geq 0,\qquad \displaystyle\frac{\partial^j}{\partial t^j}\Lambda_\epsilon\phi(t,0)=\begin{cases}\displaystyle\int_{\partial\Omega}{\phi d\sigma}& \text{ if }j=0,\\[10pt]
\displaystyle 0& \text{ if }j\geq 1;\end{cases}$
\item[(iv)] $\displaystyle \frac{\partial^j}{\partial t^j}I_\epsilon\phi(0,0)=\begin{cases}\displaystyle\int_{\Omega\setminus\om{\epsilon}}{\phi\dvol}& \text{ if }j=0,\\[10pt]
\displaystyle(-1)^j\int_{\om{\epsilon}}{\deltasr^j\phi\dvol}& \text{ if }j\geq 1;
\end{cases}$ 
\end{itemize}
where, we recall, $\om{\epsilon} = \{x\in \Omega \mid \delta(x) > \epsilon\}$.
\end{lem}
\begin{proof}
The idea of the proof is to use the divergence formula \eqref{eqn:diverg_thm} and the fact that $u_\epsilon$, with all its derivative, converges to zero, as $t\to 0$, uniformly on $\partial\Omega$. The proof is identical to the Riemannian case and we omit it. See \cite[Lemma 5.6]{Savo-heat-cont-asymp}, for details. 
\end{proof}

The next step is to obtain a small-time asymptotic expansion for the function $I_\epsilon\phi(t,0)$, cf.\  \eqref{eqn:aux_ops1}, for fixed $\epsilon>0$. Then, passing to the limit as $\epsilon\to 0$, one obtains the complete asymptotic series of $I\phi(t,0)$.

\begin{lem}
\label{lem:aux_exp1}
Under the same hypotheses of Lemma \ref{lem:iter_L}, let $\phi\in \funspace$. Then, for any $m\in\N$, we have
\begin{equation}
I\phi(t,0)=Z^{(m)}(t)+\frac{1}{\sqrt{\pi}}B^{(m)}(t)+O(t^{(m+1)/2}),\qquad\text{as } t\to 0,
\end{equation}
with
\begin{align}
Z^{(m)}(t)&=\lim_{\epsilon\to 0}\sum_{k=0}^m\frac{t^k}{k!}\int_0^\infty{e(t,r,0)L^kI_\epsilon\phi(0,r)dr},\\
B^{(m)}(t)&=\lim_{\epsilon\to 0}\sum_{k=0}^m\frac{1}{k!}\int_0^tL^k\Lambda_\epsilon\phi(\tau,0)(t-\tau)^{k-1/2}d\tau,
\end{align}
where $\displaystyle L^kI_\epsilon\phi(0,r)=\lim_{t\to 0}L^kI_\epsilon\phi(t,r)$ in the sense of distributions, cf.\ footnote in Proposition \ref{prop:iter_duhamel}.
\end{lem}

\begin{proof}
First of all, we should check that the function $I_\epsilon\phi(t,r)$ satisfies the hypotheses of Proposition \ref{prop:iter_duhamel}. This can be done employing the previous Lemmas. Indeed, thanks to Lemma \ref{lem:iter_L}, it is enough to check the hypotheses for the time derivatives of $I_\epsilon\phi$ and $\Lambda_\epsilon\phi$, and Lemma \ref{lem:hp_idp_aux} explicitly verifies them. Hence, we may apply the iterated Duhamel's principle. It remains to ensure that the last term in \eqref{eqn:iter_duhamel} is a remainder of order $(m+1)/2$, as $\epsilon\to 0$, i.e.
\begin{equation}
\lim_{\epsilon\to 0}\int_0^t{\int_0^\infty{e(t-\tau,r,0)L^{m+1}I_\epsilon\phi(\tau,r)(t-\tau)^mdr}d\tau} = O(t^{(m+1)/2}),
\end{equation}
as $t\to 0$. This can be done using again Lemma \ref{lem:hp_idp_aux} and integration by parts. 
\end{proof}

\begin{lem}
\label{lem:aux_exp2}
Under the same hypotheses of Lemma \ref{lem:iter_L}, let $\phi\in \funspace$. Then, for any $m\in\N$, as $t\to 0$, we have
\begin{multline}
\label{eqn:aux_exp2}
I\phi(t,0)=\frac{1}{\sqrt{\pi}}\sum_{k=1}^{\lfloor(m+1)/2\rfloor}\int_{\partial\Omega}Z_k\phi(y)\area t^{k-1/2}\\+\frac{1}{\sqrt{\pi}}\sum_{k=0}^{\lfloor(m-1)/2\rfloor}\int_0^tI A_k\phi(\tau,0)(t-\tau)^{k-1/2}d\tau+O(t^{(m+1)/2}),
\end{multline}
where, $Z_k$, $A_k$ are the operators (defined by compositions of $\Delta$ and $\normal$) defined in \eqref{eqn:def_Zalpha}.
\end{lem}

We omit the proof, since it is a long computation formally identical to the Riemannian case as done in \cite[Lemma 5.8]{Savo-heat-cont-asymp}. The key idea is to express the operators $Z^{(m)}(t)$, $B^{(m)}(t)$ of Lemma \ref{lem:aux_exp1} in terms of the operators $A_k,Z_k$.

\begin{thm}
\label{thm:asymp_Iphi} 
Under the same hypotheses of Lemma \ref{lem:iter_L}, let $\phi\in \funspace$. Then, for any $m\in\N$
\begin{equation}
\label{eqn:asymp_Iphi}
I\phi(t,0)=\sum_{k=1}^m\left( \int_{\partial\Omega}D_k\phi(y)d\sigma(y)\right)t^{k/2}+O(t^{(m+1)/2}),\qquad\text{as }t\to 0.
\end{equation}
where the operators $D_k$ are given in \eqref{eqn:recurs_op1}--\eqref{eqn:recurs_op3}.
\end{thm}

\begin{proof}
Let us denote by $\beta_k(\phi)$ the $k$-th coefficient in \eqref{eqn:asymp_Iphi}. We proceed by induction on $m\in\N$. The case $m=1$ is given in formula \eqref{eqn:1st_Iphi}, and the first coefficient is
\begin{equation}
\beta_1(\phi)=\sqrt{\frac{4}{\pi}}\int_{\partial\Omega}\phi d\sigma =\int_{\partial\Omega}D_1(\phi)d\sigma.
\end{equation}
Assume that \eqref{eqn:asymp_Iphi} holds for $m-1$ and for any $\phi\in \funspace$. Using Lemma \ref{lem:aux_exp2}, we recognize that the only term that we should discuss in the equality \eqref{eqn:aux_exp2} is
\begin{equation}
\sum_{k=0}^{\lfloor(m-1)/2\rfloor}\int_0^tI A_k \phi(\tau,0)(t-\tau)^{k-1/2}d\tau.
\end{equation}
By the induction hypothesis, we obtain an asymptotic expansion of $IA_k\phi$, for any $k=0,\ldots,\lfloor(m-1)/2\rfloor$, up to order $m/2$:
\begin{equation}
IA_k\phi(\tau,0)=\sum_{j=1}^{m-1}\beta_j(A_k\phi)\tau^{j/2}+O(\tau^{m/2}).
\end{equation}  
Inserting this expression in \eqref{eqn:aux_exp2} and integrating with respect to $\tau$, we obtain
\begin{multline}
\label{eqn:aux_exp3}
I\phi(t,0)=\frac{1}{\sqrt{\pi}}\sum_{k=1}^{\lfloor(m+1)/2\rfloor}\int_{\partial\Omega}Z_k\phi(y)\area t^{k-1/2}\\+\frac{1}{\sqrt{\pi}}\sum_{k=0}^{\lfloor(m-1)/2\rfloor}\sum_{j=1}^{m-1}\frac{\Gamma(j/2+1)\Gamma(k+1/2)}{\Gamma(k+(j+3)/2)}\beta_j(A_k\phi)t^{k+(j+1)/2}+O(t^{(m+1)/2}).
\end{multline}
Since, by induction hypothesis, we know that $I\phi(t,0)$ admits an asymptotic expansion up to order $m/2$, \eqref{eqn:aux_exp3} already shows that an asymptotic expansion up to order $(m+1)/2$ exists, moreover, it provides an explicit expression for the $\frac{m}{2}$-th coefficient. If $m=2n$ is even, then the first sum does not give any contribution, since the highest power of $t$ is $(m-1)/2$, while in the second sum, we have to consider only those indexes, $k$, $j$, such that the power of $t$, $k+(j+1)/2=n$, thus $j$ must be odd, i.e. $j=2i-1$, with $i=1,\ldots,n$, and $k=n-i$. Therefore, we get
\begin{equation}
\beta_{2n}(\phi)=\frac{1}{\sqrt{\pi}}\sum_{i=1}^n\frac{\Gamma(i+1/2)\Gamma(n-i+1/2)}{n!}\beta_{2i-1}(A_{n-i}\phi).
\end{equation} 
Recalling the recursive definition of $D_{2n}$ in \eqref{eqn:recurs_op2} we obtain the result for even order. On the other hand, if $m=2n+1$ is odd, then the first sum gives a contribution, whereas in the second sum $j$ must be even, i.e.\ $j=2i$, with $i=1,\dots,n$, and $k= n-i$. Hence, we obtain
\begin{equation}
\beta_{2n+1}(\phi)=\frac{1}{\sqrt{\pi}}\int_{\partial\Omega}Z_{n+1}\phi+\frac{1}{\sqrt{\pi}}\sum_{i=1}^{n}\frac{\Gamma(i+1)\Gamma(n-i+1/2)}{\Gamma(n+3/2)}\beta_{2i}(A_{n-i}\phi).
\end{equation}
Recalling the recursive definition of $D_{2n+1}$ in \eqref{eqn:recurs_op3} we obtain the result for odd order, concluding the proof.
\end{proof}
			% iterative construction of an asymptotic expansion for the localized heat content
%\input{algo.tex}

\bibliographystyle{alphaabbr}
\bibliography{biblio-heat}

\end{document}